\documentclass[oneside]{amsart}
  
\usepackage{amssymb}

\usepackage{graphicx}

\usepackage[cmtip,all]{xy}
\usepackage[latin1]{inputenc}
\usepackage{amssymb,amsmath}
\usepackage{verbatim}
\usepackage{color}
\usepackage{enumerate}
\usepackage{url}
\usepackage{geometry}
\usepackage{hyperref}

\usepackage{array}

\newtheorem{theorem}{Theorem}[section]
\newtheorem{lemma}[theorem]{Lemma}
\newtheorem{proposition}[theorem]{Proposition}

\newtheorem{corollary}[theorem]{Corollary}

\theoremstyle{definition}
\newtheorem{definition}[theorem]{Definition}

\newtheorem{assumption}[theorem]{Assumption}

\newtheorem*{thmA}{Theorem A}

\newtheorem*{theorem*}{Theorem}
\newtheorem*{proposition*}{Proposition}

\theoremstyle{remark}
\newtheorem{remark}[theorem]{Remark}

\numberwithin{equation}{section}

\newcommand\Z{\ensuremath{\mathbb Z}}\newcommand\A{\ensuremath{\mathbb A}}
\newcommand\Q{\ensuremath{\mathbb Q}}
\newcommand\C{\ensuremath{\mathbb C}}\newcommand\F{\ensuremath{\mathbb F}}

\newcommand\ab{\operatorname{ab}}

\newcommand\cont{\operatorname{cont}}

\newcommand\cycl{{\mathrm{cyc}}}

\newcommand\Frob{\operatorname{Fr}}
\newcommand\Gal{\operatorname{Gal}}
\newcommand\GL{\operatorname{GL}}

\newcommand\Ind{\operatorname{Ind}}

\newcommand\nr{\operatorname{nr}}

\newcommand\ord{\operatorname{ord}}

\newcommand\Res{\operatorname{Res}}

\newcommand\Tr{\operatorname{Tr}}

\newcommand\id{\operatorname{id}}

\newcommand{\cH}{\mathcal{S}}
\newcommand{\cC}{\mathcal{C}}

\newcommand{\fp}{{\mathfrak{p}}}

\newcommand{\fc}{{\mathfrak{c}}}

\newcommand{\cE}{{\mathcal{E}}}

\renewcommand{\L}{\mathcal{L}}
\newcommand{\tto}[1]{%
\ifthenelse{\equal{#1}{}}{\to}{\stackrel{#1}{\to}}}

\def\cO{{\mathcal O}}

\newcommand{\eps}{\varepsilon}

\newcommand{\comp}{\begin{picture}(6,5)(-3,-2)\put(0,1){\circle{2}} \end{picture}}\def\circ{\comp}

\def\res{\operatorname{res}}

\newcommand{\ra}{\rightarrow}

\newcommand{\lra}{\longrightarrow}

\def\XXint#1#2#3{{\setbox0=\hbox{$#1{#2#3}{\int}$}
\vcenter{\hbox{$#2#3$}}\kern-.5\wd0}}

\newcommand{\cG}{\mathcal{G}}
\newcommand{\bff}{{\bf f}}
\newcommand{\bfg}{{\bf g}}
\newcommand{\bfh}{{\bf h}}
\newcommand{\ubff}{\breve{\bf f}}
\newcommand{\ubfg}{\breve{\bf g}}
\newcommand{\ubfh}{\breve{\bf h}}

\newcommand\cW{\ensuremath{\mathcal W}}
\newcommand\coh{\operatorname{H}}

\newcommand\fil{\operatorname{Fil}}
\newcommand\norm{\operatorname{N}}

\newcommand{\derham}{\mathrm{dR}}
\newcommand{\cris}{{\mathrm{cris}}}

\newcommand{\V}{\mathbb{V}}

\newcommand{\Sel}{\operatorname{Sel}}

\newcommand{\Sh}{\operatorname{Sh}}
\newcommand{\Tate}{\mathrm{Tate}}
\newcommand{\bal}{\mathrm{bal}}

\newcommand{\Dderham}{\operatorname{D}_\mathrm{dR}}

\newcommand{\Art}{\mathrm{Art}}
\newcommand{\anti}{\mathrm{ant}}
\newcommand{\cR}{\mathcal{R}}
\newcommand{\inv}{{-1}}
\newcommand{\BB}{\mathrm{B}}

\usepackage{listings}
\renewcommand{\hom}{\mathrm{Hom}}

\begin{document} 

\title[Kolyvagin classes vs diagonal classes]{Kolyvagin classes versus non-cristalline diagonal classes}

\author{Francesca Gatti and Victor Rotger}
\address{F. G.: Departament de Matem\`{a}tiques, Universitat Polit\`{e}cnica de Catalunya, C. Jordi Girona 1-3, 08034 Barcelona, Spain}
\email{francesca.gatti@upc.edu}

\address{V. R.: IMTech, UPC and Centre de Recerca Matem\`{a}tiques, C. Jordi Girona 1-3, 08034 Barcelona, Spain}
\email{victor.rotger@upc.edu}

\thanks{This project has received funding from the European Research Council (ERC) under the European Union's Horizon 2020 research and innovation programme (grant agreement No 682152). 
The second author gratefully acknowledges Icrea for financial support through an Icrea Academia award. We are most grateful to Adrian Iovita and Xevi Guitart for their interest and extremely helpful comments and to Matteo Longo for reading our work in detail.
}

\date{\today}


\begin{abstract}
Let  $E/\Q$ be an elliptic curve having multiplicative reduction at a prime $p$. Let  $(g,h)$ be a pair of eigenforms of weight $1$ arising as the theta series of an imaginary quadratic field $K$, and assume that the triple-product $L$-function $L(f,g,h,s)$ is self-dual and does not vanish at the central critical point $s=1$. The main result of this article is a formula expressing the $p$-adic iterated integrals introduced in \cite{DLR} to the Kolyvagin classes associated by Bertolini and Darmon to a system of Heegner points on $E$. 
\end{abstract}
\maketitle 


\section{Introduction}\label{sec: Introduction}

Let $E/\Q$ be an elliptic curve with square-free conductor $N_f$ and let $f\in S_2(N_f)$ be the normalized newform attached to it by modularity. Let $K$ be an imaginary quadratic field of discriminant $-D_K$ relatively prime to $N_f$. Let $$\psi_g,\psi_h:\A_K^\times/K^\times\longrightarrow\C^\times$$ be two finite order Hecke characters of $K$ of conductors $\fc_g,\fc_h$ respectively and let $g:=\theta(\psi_g)$ and $h:=\theta(\psi_h)$ denote the theta series associated to them. Assume that the central characters  of $g$ and $h$ are mutually inverse. Then
\begin{equation*}
	g\in M_1(N_g,\chi), \qquad h\in M_1(N_h,\chi^{-1})
\end{equation*}
are weight one modular forms of level $N_g=D_K\cdot N_{K/\Q}(\fc_g), N_h=D_K\cdot N_{K/\Q}(\fc_h)$ respectively. 
For simplicity we assume throughout that 
\begin{equation}\label{eq: assumption on levels}
	\gcd(N_f,N_g N_h)=1
\end{equation}
and we set $N := \mathrm{lcm}(N_f,N_g,N_h)$.
Let $p$ be an odd prime number such that 
\begin{equation}\label{eq: conditions on p}
	p\mid\mid N_f \ \text{ and } \ p \text{ is inert in $K$}.
\end{equation} 
Hence $E$ has multiplicative reduction at $p$ and the completion $K_p$ of $K$ at $p$ is the unramified quadratic extension of $\Q_p$. Let 
\begin{equation*}\label{eq: tates uniformisation in the text}
	\varphi_\Tate:\bar{\Q}_p^\times/q_E^\Z\overset{\cong}{\longrightarrow}E(\bar{\Q}_p) 
\end{equation*}
be Tate's uniformization, with $q_E\in p\Z_p$, and assume throughout that 
\begin{equation}\label{ass: p doesnt divide valuation of tates uniformizer}	p\nmid\ord_p(q_E).
\end{equation}



Set $$a:=a_p(E)=a_p(f) \in \{ \pm 1\}.$$ In other words, $a=1$ (resp.\,$a=-1$) according to whether $E$ has split (non-split) multiplicative reduction at $p$. 

Let $\alpha_g,\beta_g$ (resp.\,$\alpha_h,\beta_h$) denote the roots of the $p$-th Hecke polynomial of $g$ (resp.\,of $h$).  Note that \eqref{eq: conditions on p} implies that \begin{equation*}
	\alpha_g = -\beta_g \ \text{ and } \ \alpha_h = -\beta_h.
\end{equation*} 
Since the nebentype character of $h$ is the inverse of that of $g$, one either has
\begin{equation*}\label{eq: two cases for the eigenvalues of h}
	(\alpha_h,\beta_h)=(1/\alpha_g,-1/\alpha_g) \ \text{ or } \ (\alpha_h,\beta_h)=(-1/\alpha_g,1/\alpha_g).
\end{equation*}
Throughout this article we fix the ordering of the pair $(\alpha_h,\beta_h)$ in such a way that
\begin{equation}\label{ah-bh}
\alpha_g \cdot \alpha_h = -a.
\end{equation}

Fix throughout algebraic closures $\bar\Q$ and $\bar\Q_p$ of $\Q$ and $\Q_p$ respectively and an embedding $\bar\Q \hookrightarrow \bar\Q_p$. Let $L\subset \bar\Q$ denote the number field generated by the traces of $\psi_g$ and $\psi_h$, together with the roots $\alpha_g, \beta_g, \alpha_h, \beta_h$. 
and let  $L_p$ denote the completion of $L$ within $\bar\Q_p$.

Set $g_\alpha(z):=g(z)-\beta_g g(pz)$ and define $h_\alpha$ analogously. Let $\bff,\bfg_\alpha,\bfh_\alpha$ be Hida families passing through $f, g_\alpha,h_\alpha$ respectively. As explained in \cite{DR1} and \cite{DR2},  associated to any choice of $\Lambda$-adic test vectors $(\breve{\bff},\breve{\bfg}_\alpha,\breve{\bfh}_\alpha)$ of tame level $N$  there is a three-variable $p$-adic $L$-function 
\begin{equation*}
\L_p^g(\breve{\bff},\breve{\bfg}_\alpha,\breve{\bfh}_\alpha).
\end{equation*}
This $p$-adic $L$-function interpolates the square-roots of the central values of the classical $L$-function $L(\breve{f}_k\otimes \breve{g}_\ell\otimes \breve{h}_m,s)$ attached to the specializations of the choice of test vectors at classical points of weights $k,\ell,m$  with $k,\ell,m\geq 2$ and $\ell\geq k+m$. 

Note that the point $(2,1,1)$ corresponding to our triple of modular forms $(f,g,h)$ lies outside the region of classical interpolation for $\L_p^g(\breve{\bff},\breve{\bfg}_\alpha,\breve{\bfh}_\alpha)$.
As in \cite{DLR} and \cite{GGMR19} we are interested in studying the $p$-adic $L$-value
\begin{equation*}\label{iter}
I_p(f,g_\alpha,h_\alpha) := \L_p^g(\breve{\bff},\breve{\bfg}_\alpha,\breve{\bfh}_\alpha)(2,1,1)
\end{equation*}
regarded as an element of $\bar\Q_p^\times/L^\times$. As explained in \cite{DR2} and \cite{DLR}, different choices of $\Lambda$-adic test vectors yield the same $p$-adic value $I_p(f,g_\alpha,h_\alpha)$ up to algebraic scalars  in $L^\times$. This 
entitles us to denote it simply $I_p(f,g_\alpha,h_\alpha)$, dropping from the notation this choice.

The triple-product $L$-function $L(f,g,h,s) = L(E,\rho,s)$ may be recast as the Hasse-Weil $L$-series of the twist of $E$ by the tensor product $\rho=\rho_g\otimes \rho_h$ of the two Artin representations associated to $g$ and $h$. In light of \eqref{eq: assumption on levels},  the order of vanishing of $L(f,g,h,s)$ at $s=1$ is always even and one thus expects the central critical value $L(E,\rho,1)$ to be generically nonzero. We assume throughout that we fall indeed in this generic case, that is to say, we have
\begin{assumption}\label{ass: rank 0}
	$L(E,\rho,1) \neq0$.
\end{assumption}

The main purpose of this article is proving a formula relating the $p$-adic iterated integral  $I_p(f,g_\alpha,h_\alpha) $ to the Kolyvagin classes associated by Bertolini and Darmon in \cite{BD97}.

In order to recall briefly the latter,  for every $m\geq 1$ let $F_m/H$ be the layer of degree $p^m$ within the anticyclotomic $\Z_p$-extension of $H$. In \cite{BD97} Bertolini and Darmon associated to a canonical collection of {\em Heegner points} $\{\alpha_m\in E(F_m)\}$ a global cohomology class $\mathbf{K}\in\coh^1(H,V_f)$ that is referred to in loc.\,cit.\,as the {\em Kolyvagin class} associated to $E/H$.
 


Fix a prime $\fp$ of $H$ above $p$ and denote $\Phi_{m,\fp}$ the $p$-primary part of the group of connected components of the N\'eron model of $E$ over the completion of $F_m$ at the unique prime ideal above $\fp$. Denote $\Phi_{\infty,\fp}:=\varprojlim_m\Phi_{m,\fp}$ the projective limit with respect to the natural projection maps $\Phi_{m,\fp} \ra \Phi_{m-1,\fp}$.

As explained in Lemma \ref{lem: Phi infty and limit of bar Phi1 + Phi infty and limit of bar Phi2 + rem: 2} there is a canonical isomorphism $\varphi:\Phi_{\infty,\fp}\overset{\cong}{\longrightarrow}\Z_p$ induced by Tate's uniformization. 
Denote by $\bar{\alpha}_m$ the image of $\alpha_m$ in $\Phi_{m,\fp}$ and set $\bar{\alpha}:=(\bar{\alpha}_m)_m\in\Phi_{\infty,\fp}.$ Define the period
$$
\Pi_p := \varphi(\bar{\alpha}) \in \Z_p.
$$

Note that $H^1(K_p,V_f)$ is naturally a $\Gal(K_p/\Q_p)$-module and we let $H^1(K_p,V_f)^\pm$ denote the $\Q_p$-subspace on which $\Frob_p$ acts with eigenvalue $\pm 1$.

The Kolyvagin class $\mathbf{K}$ is not in general cristalline at $p$, that is to say, the local class $\res_p(\mathbf{K}) \in H^1(K_p,V_f)$ is not expected to lie in Bloch-Kato's finite subspace $H^1_{f}(K_p,V_f)$. Nevertheless, one can show (cf.\,Proposition \ref{prop: exp partialp C}) that $\res_p(\mathbf{K})^a := \res_p(\mathbf{K})+a\res_p(\mathbf{K}^{\Frob_p})$ does lie in $H^1_{f}(K_p,V_f)^a$ and therefore there exists a local point
\begin{equation*}
Q_p^a \in (E(K_p)\otimes \Q_p)^a \quad \mbox{ such that } \quad \delta_p(Q_p^a) = \res_p(\mathbf{K})^a,
\end{equation*}
where $\delta_p: E(K_p)\otimes \Q_p \stackrel{\sim}{\lra} H^1_{f}(K_p,V_f)$ stands for Kummer's isomorphism. In spite of the notation we have chosen, beware that $Q_p^a$ is {\em not} expected to be the $a$-component of any local point $Q_p\in E(K_p)\otimes \Q_p$, precisely because $\res_p(\mathbf{K})$ does not lie in $H_f^1(K_p,V_f)$. In other words, while $\res_p(\mathbf{K})^a$ is cristalline, the local class $\res_p(\mathbf{K})^{-a}$ is not.


 

Let $c= \prod_{v\nmid N \infty} c_v(\breve{f},\breve{g}_\alpha,\breve{h}_\alpha) \in L$ denote the product of local automorphic factors appearing in \cite[Prop.\,2.1\,(a)\,(iii)]{DLR} associated to the choice of test vectors at $(2,1,1)$,
and define the algebraic $L$-value $$L^{\mathrm{alg}}(E\otimes\rho,1) := \dfrac{L(E\otimes\rho,1)}{\pi^4\langle f, f\rangle^2}.$$ 
It follows from the work of Harris and Kudla \cite{HK91} that the above ratio lies in $L$ and is in fact non-zero by \eqref{ass: rank 0}.
The following is the main theorem of this note.
 
\begin{thmA} For any choice of test vectors, we have
\begin{equation*}
I_p(f,g_\alpha,h_\alpha)=
\dfrac{\sqrt{c} \cdot \sqrt{L^{\mathrm{alg}}(E\otimes\rho,1)}}{\Pi_p \cdot \L_{g_\alpha}} \times \log_p(Q_p^a) \qquad (\mod L^\times)
\end{equation*}
where 
\begin{itemize}
\item $\L_{g_\alpha}\in K_p$ is a period on which $\Frob_p$ acts as multiplication by $-1$ and only depends on $g_\alpha$, 
\item $\log_p:E(K_p)\longrightarrow K_p$ denotes the $p$-adic logarithm.
\end{itemize}
\end{thmA}

\begin{remark} Period $\L_{g_\alpha}$ was first introduced in \cite{DR2b} and is well-defined only up to scalars in $L^\times$; cf.\,\eqref{Lga} for more details.
Let $H_g$ be the number field cut out by the Artin representation attached to the adjoint of $g$ and fix a completion $H_{g,p}$ of $H_g$ at a prime above $p$. In   \cite{DR2b}  it is conjectured that $\L_{g_\alpha} = \log_p(u_{g_\alpha}) \,\, (\mathrm{mod}\,  L^\times)$, where $\log_p:\cO^\times_{H_{g,p}}\otimes L\longrightarrow H_{g,p}\otimes L$   is the usual $p$-adic logarithm and $u_{g_\alpha}\in\cO_{H_g}[1/p]^\times\otimes L$  is the so-called Gross--Stark unit attached to $g_\alpha$ \cite[\S1.2]{DLR}. See \eqref{Lga} for the explicit definition of $\L_{g_\alpha}$.
\end{remark}

The body of the paper is devoted to the proof of this result. In \S \ref{sec: ex appendix} we relate Bloch-Kato's dual exponential map on $E$ to the group of connected components. This is an exercise in $p$-adic Hodge theory which is surely well-known to experts, but we included because we did not find precise references in the literature. We are most grateful to A. Iovita for his assistance in this section. In \S \ref{sec: the selmer group of fxgxh} we review the theory of Hida families and Selmer groups, and culminates with the description of an explicit basis of the relaxed Selmer group associated to the triple $(f,g,h)$. In \S \ref{sec: main thm 1} we exploit the fine work \cite{BSV19reciprocity} of Bertolini, Seveso and Venerucci in order to prove a formula relating the iterated integral $I_p(f,g_\alpha,h_\alpha)$ in terms of the basis described in the previous section. Finally, in the last section we introduce Kolyvagin classes and prove our main result, which is found in Corollary \ref{cor: p inert explicit version nonsplit}.

Theorem A is vaguely similar to the main theorem of \cite[\S 3]{DLR}, although both the statements and specially the proofs are ostensibly different. While in loc.\,cit.\,the motivic elements appearing in the statement (both the diagonal classes and Heegner points) are cristalline at $p$, in our article they are not and this makes the calculations substantially different. Moreover, in \cite[\S 3]{DLR} one makes crucial use of a factorization of $p$-adic $L$-series (which follows rather trivially from a comparison of critical $L$-values), while in our setting the analogous $p$-adic $L$-functions associated to $K$ are not available (because in our scenario $p$ remains inert in $K$). It would be of great interest to investigate whether the recent $p$-adic $L$-functions of Andreatta and Iovita \cite{AI19} could be exploited in order to provide an alternative proof of our formula, although this does not appear to be a straight-forward project.

As we have already pointed out, one of the key ingredients of our proof are the results of Bertolini, Seveso and Venerucci in \cite{BSV19reciprocity}. In loc.\,cit.\,the authors study the exceptional zero phenomena appearing in the scenario of diagonal cycles at weights $(2,1,1)$ when  $\alpha_g \alpha_h = \pm a$. An arithmetic application of these results is obtained in \cite{BSV19balanced} and  \cite{DR19Stark} to the theory of Stark-Heegner points on elliptic curves,  in the setting where $\alpha_g \alpha_h = a$ and an {\em improved} $p$-adic $L$-function plays a relevant role. The present work focus instead in the opposite setting where $\alpha_g \alpha_h = -a$ and it is rather an {\em improved} diagonal cohomology class (cf.\,\cite[\S 8.3]{BSV19reciprocity}  and \S \ref{sec: exceptional cases} of this note) which plays a prominent role.

\section{$p$-adic Hodge theory for elliptic curves with multiplicative reduction}\label{sec: ex appendix}

Let $E/\Q_p$ be an elliptic curve of multiplicative reduction at a prime $p$. Let
\begin{equation*}
\chi_E: G_{\Q_p}\longrightarrow\{ \pm1  \}
\end{equation*}
be the trivial character if $E$ has split multiplicative reduction over $\Q_p$, and  the quadratic unramified character if $E$ has non-split multiplicative reduction. For any $G_{\Q_p}$-module $M$, let $M(\chi_E)$ denote the twist of $M$ by $\chi_E$. 

Tate's uniformisation provides a $G_{\Q_p}$-equivariant isomorphism 
\begin{equation}\label{eq: tate's uniformisation}
\varphi_\Tate:\bar{\Q}_p^\times(\chi_E)/q_E^\Z \overset{\cong}{\longrightarrow} E(\bar{\Q}_p)
\end{equation}
for some $q_E\in p\Z_p$. 
As in \eqref{ass: p doesnt divide valuation of tates uniformizer}, assume throughout that $p\nmid n:=\ord_p(q_E).$



Let $T = T_E:=\varprojlim E[p^m]$ denote the Tate module associated to $E$ and set $V:=T\otimes{\Q_p}$. Let $E_0(\bar{\Q}_p)$ denote the set of points in $E(\bar{\Q}_p)$ that stay nonsingular in the special fiber of $E$ at $p$. The module 
\begin{equation*}
	T^+:=\varprojlim E_0[p^m]
\end{equation*}
fits in an exact sequence of $\Z_p[G_{\Q_p}]$-modules
\begin{equation}\label{eq: filtration of T} 
0\longrightarrow T^+\overset{\iota}{\longrightarrow} T\overset{\pi}{\longrightarrow}T^-\longrightarrow0,
\end{equation}
where $T^-:=T/T^+$. 
The modules $T^+$ and $T^-$ are free over $\Z_p$ of rank one, and Tate's uniformization \eqref{eq: tate's uniformisation} induces $G_{\Q_p}$-equivariant isomorphisms
\begin{equation}\label{eq: explicit Tf+, Tf-}
	\Z_p(\chi_E)(1)\overset{\cong}{\longrightarrow}T^+, \qquad  \Z_p(\chi_E)=\varprojlim_m(\bar{\Q}_p^\times(\chi_E)/q_E^\Z)[p^m]/\mu_{p^m}(\chi_E)\overset{\cong}{\longrightarrow}T^-.
\end{equation}
More precisely, if we fix compatible systems
\begin{equation*}
\tilde{\eps}:=(\tilde{\eps}^{(m)})_m\in\Z_p(\chi_E)(1), \qquad \tilde{q}:=(q_E^{1/p^m})_m\in\varprojlim\left(\bar{\Q}^\times_p/q_E^\Z\right)[p^m]
\end{equation*}
then
\begin{equation*}
\eps:=\varphi_{\Tate}(\tilde{\eps}), \ q:=\varphi_{\Tate}(\tilde{q})
\end{equation*}
form a $\Z_p$-basis of $T$. Moreover, if  $\chi_{\cycl}:G_{\Q}\longrightarrow\Z_p^\times$ denotes the cyclotomic character, then
\begin{enumerate}
	\item $\eps$ is a basis of $T^+$ on which $G_{\Q_p}$ acts as the character $\chi_E\chi_{\cycl}$;
	\item $\bar{q}:=\pi(q)$ is a basis of $T^-$ on which $G_{\Q_p}$ acts via $\chi_E$.
\end{enumerate}

\bigskip

Let $K$ be an imaginary quadratic field in which $p$ is inert and let  $K_p$ denote the completion of $K$ at $p$. Notice that ${\chi_E}_{|G_{K_p}}=1$ and thus 
\eqref{eq: filtration of T} gives an exact sequence of $\Q_p[G_{K_p}]$-modules
\begin{equation}\label{eq: filtration of V of Kp}
0\longrightarrow V^+\overset{\iota}{\longrightarrow}V\overset{\pi}{\longrightarrow}V^-\longrightarrow 0
\end{equation}
such that $\dim_{\Q_p}V^+=\dim_{\Q_p}V^-=1$,  $G_{K_p}$ acts on $V^+$ via $\chi_{\cycl}$, and $G_{K_p}$ acts trivially on $V^-$.

\subsection{The group of connected components}\label{sec: tate module of E and group of connected component}

Fix a ring class field $H$ of $K$ of conductor $c$ prime to $p$ and for every integer $m\geq 1$, let $H(p^m)$ be the ring class field of $K$ of conductor $c\cdot p^m$. The Galois group $\Gal(H(p^m)/H)$ is cyclic of order $e_m:=(p+1)p^{m-1}$. Let $F_m$ be the intermediate field $H\subseteq F_m\subseteq H(p^{m+1})$ such that $\Gal(F_m/H)$ is cyclic of order $p^m$. Since $p$ is inert in $K$, the prime ideal $p\cO_K$ splits completely in $H$. Fix once and for all a prime $\fp$ of $H$ above $p$; it ramifies totally in $H(p^m)$ as $\fp\cO_{F_m}=\fp_{m}^{p^m}$.  Let us denote $F_{m,\fp}$ the completion of $F_m$ at $\fp_{m}$, $\cO_{m,\fp}$ its ring of integers and $\mathbb{F}_{m,\fp}$ its residue field. 

Let $\cE$ be the N\'eron model of $E$ over $\Z_p$, and let $\tilde{\cE}:=\cE\times_{\Z_p}\F_p$ denote its special fiber.   For all $m\geq 1$, let $\Phi_{m,\fp}$ denote the $p$-Sylow subgroup of the group \begin{equation}\label{eq: ex Phimp}
\tilde{\cE}(\F_{m,\fp})/\tilde{\cE}_0(\F_{m,\fp})\cong E(F_{m,\fp})/E_0(F_{m,\fp})
\end{equation} of connected components of the base change of  $\cE$ to  $\cO_{m,\fp}$.

%
%
%

\begin{lemma}\label{lem: Phi infty and limit of bar Phi1 + Phi infty and limit of bar Phi2 + rem: 2}
There are  isomorphisms of $G_{\Q_p}$-modules
	\begin{equation*}\label{eq: the first overline varphi tate}
		T^- \, \, \cong \, \,  \Z_p(\chi_E) \, \, \cong \, \,	\Phi_{\infty,\fp}.
	\end{equation*} 
\end{lemma}
\begin{proof} The first isomorphism is \eqref{eq: explicit Tf+, Tf-}.
	Tate's uniformisation provides a description of the group \eqref{eq: ex Phimp} as
	\begin{equation}\label{eq: iso phi m fp}
	\overline{\varphi}_\Tate: 	F_{m,\fp}^\times/q_E^\Z \cO_{m,\fp}^\times(\chi_E)\cong\left(\Z/\ord_{m,\fp}(q_E)\Z\right)(\chi_E) \, \overset{\cong}{\longrightarrow} \,  E(F_{m,\fp})/E_0(F_{m,\fp}),
	\end{equation}
	where $\ord_{m,\fp}$ is the discrete valuation on $\cO_{m,\fp}$. 
	We have  $p\cO_{F_{m}}=\prod_{\fp_{m}\mid p}\fp_m^{p^m}$, where $\fp_m=(\pi_m)$ is the maximal ideal of $\cO_{m,\fp}$. Hence $q_E=p^n\alpha= \pi_m^{np^m}\alpha'$ where $\alpha,\alpha'\in\cO_{F_{\fp,m}}^\times$, i.e.
	\begin{equation*}
	\ord_{m,\fp}(q_E)=n\cdot p^m.
	\end{equation*}
      Under condition \eqref{ass: p doesnt divide valuation of tates uniformizer}, then \eqref{eq: iso phi m fp} gives the isomorphism 
	\begin{equation*}\label{eq: description of Phimp}
	\overline{\varphi}_\Tate: \left(\Z/np^m\Z\right)(\chi_E)\cong \left(\Z/n\Z\right)(\chi_E) \oplus  \left(\Z/p^m\Z\right)(\chi_E) \, \overset{\cong}{\longrightarrow} \, E(F_{m,\fp})/E_0(F_{m,\fp}). 
	\end{equation*}
	The lemma follows after taking $p$-primary parts and passing to the inverse limit.
	
\end{proof}

\subsection{The $G_{K_p}$-cohomology of $E$}\label{sec: the GKp cohomology of E}

 Let $\BB_\derham=\mathrm{Frac}(\BB_\derham^+)$ denote Fontaine's algebra of de Rham periods. For a de Rham $G_{K_p}$-representation $W$, denote 
 \begin{equation*}
 \Dderham(W):=(W\otimes_{\Q_p}\BB_\derham)^{G_{K_p}}.
 \end{equation*}
 This Dieudonn\'e module comes naturally equipped with a filtration, inherited from the filtration on $\BB_\derham$, i.e.\
 \begin{equation*}
 \fil^j\Dderham(W):=(W\otimes_{\Q_p} \fil^j\BB_\derham)^{G_{K_p}},
 \end{equation*}
 and in particular we have $\fil^0\Dderham(W)=(W\otimes \BB_\derham^+)^{G_{K_p}}$. The inclusion $\alpha:W\longrightarrow W\otimes_{\Q_p} \BB_\derham$ 
 induces in cohomology the map 
 $$
 \alpha_*:\coh^1(K_p,W)\longrightarrow\coh^1(K_p, W\otimes \BB_\derham).
 $$ 
 Define
 \begin{equation*}
 \coh^1_g(K_p,W):=\ker(\alpha_*)\subseteq \coh^1(K_p,W), \qquad \coh^1_s(K_p,W):=\coh^1(K_p,W)/\coh^1_g(K_p,W).
 \end{equation*}
 Similarly, if $\BB_\cris$ denotes the ring of cristalline periods, define 
 \begin{equation*}
 	\coh^1_f(K_p,W):=\ker\left( \coh^1(K_p,W)\longrightarrow \coh^1(K_p,W\otimes\BB_\cris) \right).
 \end{equation*}


Let  $\log_p:1+p\Z_p\longrightarrow p\Z_p$ be the usual $p$-adic logarithm, $\exp_p$ denote its inverse and set  $u:=\exp_p(p) \in 1+p\Z_p$.   Let 
\begin{equation}\label{Artin}
\Art:K_p^\times\longrightarrow G_{K_p}^{\ab}
\end{equation}
denote the Artin map of local class field theory, where $G_{K_p}^{\ab}$ denotes the Galois group of the maximal abelian extension $K_p^{\ab}$ of $K_p$ over $K_p$. Recall that \eqref{Artin} gives an isomorphism once extended to the profinite completion $\widehat{K}_p^\times$ of $K_p^\times$.

Decomposition
\begin{equation*}\label{eq: decomposition Kptimes}
K_p^\times= p^\Z \oplus (1+p\cO_{K_p}) \oplus \left( \cO_{K_p}/p\cO_{K_p} \right)^\times,
\end{equation*}
corresponds via \eqref{Artin}  to
\begin{equation*}
G_{K_p}^{\ab}\cong \Gal(K_p^{\nr}/K_p) \times \Gal(K_{\infty}/K_p) \times \Gal(K_p'/K_p)
\end{equation*}
where $K_p^{\nr}$ is the maximal  unramified extension of $K_p$, $K_{\infty}/K_p$ is a $\Z_p^2$-extension and $K_p'/K_p$ is finite. More precisely, via the Artin map, $p^{\widehat{\Z}}\cong\Gal(K_p^{\nr}/K_p)$.
Fix a prime $\fp$ of $H$ above $p$, and denote
\begin{equation*}
K_{p}(\mu_{p^\infty}):=\underset{\longrightarrow}{\lim} \ K_p(\mu_{p^m}); \qquad H(p^{\infty})_\fp:=\underset{\longrightarrow}{\lim} \ H(p^m)_{\fp_m},
\end{equation*}
where $\fp_m$ is the unique prime of $H(p^m)$ lying above $\fp$. 
Let  
\begin{equation*}
K_p\subseteq K_{\cycl,p}\subseteq K_p(\mu_{p^\infty}), \qquad K_p\subseteq K_{\anti,p}\subseteq H(p^\infty)_\fp
\end{equation*}
be maximal sub-extensions such that $\Gal(K_{\cycl,p}/K_p)\simeq \Gal(K_{\anti,\fp}/K_p) \simeq \Z_p$. Then 
\begin{equation*}
K_{\infty}=K_{\cycl,p}\cdot K_{\anti,p}.
\end{equation*} 
Fix  an element $u_{\star}\in1+p\cO_{K_p}$ such that $\Frob_p{u}_{\star}=-u_{\star}$ and $\{u,u_{\star}\}$ is a $\Z_p$-basis of $1+p\cO_{K_p}$. Then
\begin{equation*}
\Gal(K_{\infty,\fp}/K_p)\cong\Gamma_{\cycl}\times \Gamma_{\anti}, 
\end{equation*}
where
\begin{enumerate}
	\item  $\Gamma_{\cycl}:= \Gal(K_{\cycl,p}/K_p)$ is generated topologically by $\sigma_{\cycl} := \Art(u)$;
	\item $\Gamma_{\anti} :=\Gal(K_{\anti,p}/K_p)$ is generated topologically by $\sigma_{\anti} := \Art(u_{\star})$.
\end{enumerate}

Set also $\Gamma_{\nr}:=\Gal(K_p^{\nr}/K_p)$. Recall  the cyclotomic character $\chi_{\cycl}: G_{\Q} \ra \Z_p^\times$; by an abuse of notation, we continue to denote $\chi_{\cycl}: \Gamma_{\cycl} \lra 1+p\Z_p$ its restriction to $\Gamma_{\cycl}$.
Let also $\chi_{\anti}$ and $\chi_{\nr}$ denote the characters of $\Gamma_{\anti}$ and $\Gamma_{\nr}$ such that  $$\chi_{\anti}(\sigma_\anti)= u, \quad \chi_{\nr}(\Frob_p) = u$$
respectively.

Note that $$\{   {\xi}_{\nr} := \log_p(\chi_{\nr}), {\xi}_{\cycl} := \log_p(\chi_{\cycl}), {\xi}_{\anti} := \log_p (\chi_{\anti}) \}$$
is naturally a $\Q_p$-basis of $\hom_{\cont}(G_{K_p}^{\ab},{\Q_p})$. After composing with the isomorphism  $\Q_p\overset{\overline{\varphi}_\Tate}{\longrightarrow} V_{f|G_{K_p}}^-$ provided by Tate's uniformization this further yields a $\Q_p$-basis
of $\coh^1(K_p,V^-)$, which in turn may be regarded as a basis of  $\coh^1(K_p,\Phi_{\infty,\fp}\otimes\Q_p)$ by means of 
Lemma \ref{lem: Phi infty and limit of bar Phi1 + Phi infty and limit of bar Phi2 + rem: 2}.

For a class $\xi\in \coh^1(K_p,V^-)$,  we denote  $\bar{\xi}$  its image in the singular quotient $\coh^1_s(K_p,V^-)$. 

\begin{lemma}\label{lemma: 2}

		 $\{\bar{\xi}_{\cycl}, \bar{\xi}_{\anti}\}$ is a $\Q_p$-basis for $\coh^1_s(K_p,V^-)$.

\end{lemma}

\begin{proof}
By definition, the submodule $\coh^1_g(K_p,{\Q_p})=\coh^1_f(K_p,{\Q_p})$ of $\coh^1(K_p,{\Q_p})=\hom_{\cont}(G_{K_p},\Q_p)$ is given by $\coh^1_f(K_p,{\Q_p}(\chi_E))\cong\hom_{\cont}(\Gamma_{\nr},{\Q_p})$ and is generated by $\xi_{\nr}$. The lemma follows because $\coh^1_s(K_p,{\Q_p}) = \coh^1_s(K_p,V^-) = \coh^1(K_p,V^-)/\coh^1_f(K_p,V^-)$.
\end{proof}

Write the uniformizer of the elliptic curve $E/\Q_p$ as
\begin{equation}\label{eq: notation for qE}
q_E= p^nu^sx\in p\Z_p, \qquad n\geq1, s\in\Z_p, x\in\mu_{p-1},
\end{equation}
so that $n=\ord_p(q_E)$ and $ps=\log_p(q_E)$.
Let
\begin{equation*}\label{eq: projection Vf Vf- in cohomology}
\pi_*:\coh^1(K_p,V)\longrightarrow\coh^1(K_p,V^-)	
\end{equation*} be the morphism induced in $G_{K_p}$-cohomology by the projection $\pi:V\longrightarrow V^-$.


For a $\Gal(K_p/\Q_p)$-module $M$, let us write  $M^\pm$ for the eigenspace on which $\Frob_p$ acts as $\pm1$. Set $a:=a_p(E)$ as in the introduction.

\begin{proposition}\label{prop:1}

	Let 
	$x_\cycl,x_\anti\in\coh^1(K_p,V)$  be elements such that 
	\begin{equation}\label{eq: image of xcyc e xant}
	\pi_*(x_\cycl)=n\xi_{\cycl}-s\xi_{\nr}, \qquad \pi_*(x_\anti)=\xi_{\anti} \, \mbox{ in } \, \coh^1(K_p,V^-).
	\end{equation}
	Then $\{\bar{x}_\cycl,\bar{x}_\anti\}$ is a $\Q_p$-basis for $\coh^1_s(K_p,V)$ and $\pi_*$ descends to an isomorphism 
		\begin{equation*}
		\bar{\pi}_*:\coh^1_s(K_p,V)\overset{\cong}{\longrightarrow}\coh^1_s(K_p,V^-), \qquad \bar{x}_\cycl\mapsto n\cdot\bar{\xi}_{\cycl}, \qquad  \bar{x}_\anti\mapsto\bar{\xi}_{\anti}.
		\end{equation*}
Moreover, the Frobenius eigenspaces $\coh^1_s(K_p,V)^{a}$  and $\coh^1_s(K_p,V)^{-a}$ are generated respectively by $\bar{x}_{\cycl}$ and $\bar{x}_{\anti}$.	
\end{proposition}
\begin{proof}
	Consider  the long exact sequence 
	\begin{equation}\label{eq: les}
	0\rightarrow V^-\overset{\partial^0}{\longrightarrow}\coh^1(K_p,V^+)\overset{\iota_*}{\longrightarrow}\coh^1(K_p,V)\overset{\pi_*}{\longrightarrow}\coh^1(K_p,V^-)\overset{\partial^1}{\longrightarrow}\coh^2(K_p,V^+)\rightarrow0
	\end{equation}
	induced in $G_{K_p}$-cohomology by \eqref{eq: filtration of V of Kp}. The  connecting homomophisms $\partial^0,\partial^1$ can be expressed in terms of $q_E$ as follows.
	Cup product and the trace map give a pairing
	\begin{equation*}
		\langle \ , \ \rangle: \coh^1(K_p,V^+)\times\coh^1(K_p,V^-)\longrightarrow\coh^2(K_p,V^+)\cong\Q_p
	\end{equation*}
satisfying
\begin{equation*}
	\langle \delta_p(q_E),\xi\rangle =\xi(\Art(q_E))
\end{equation*}
for all $\xi\in\coh^1(K_p,V^-)$. 
If we still call $\partial^1$ its composition with  $\coh^2(K_p,V^+)\cong \Q_p$, then it 
coincides with the map $\langle \delta_p(q_E), \cdot \rangle$. Using the notation of \eqref{eq: notation for qE},  we have $\Art(q_E)=\Frob_p^n\sigma_{\cycl}^{s}$ and thus
	\begin{enumerate}
		\item $\partial^1(\xi_{\cycl})=\xi_{\cycl}(\Art(q_E))=\xi_{\cycl}(\sigma_{\cycl}^{s})=\log_p(\chi_{\cycl}(\sigma_{\cycl}^{s})=s\log_p(u)=sp;$
		\item 	$\partial^1(\xi_{\anti})=\xi_{\anti}(\Art(q_E))=\log_p(\chi_{\anti}(\sigma_{\anti}^0))=0;$
		\item $	\partial^1(\xi_{\nr})=\xi_{\nr}(\Art(q_E))=\xi_{\nr}(\Frob_p^n)=\log_p(\chi_{\nr}(\sigma_{\cycl})^n)=n\log_p(u)=np.$
	\end{enumerate}

Combining these computations with the exactness of \eqref{eq: les} we deduce that the image of $\pi_*$ is generated by $\{ n\cdot\xi_{\cycl}-s\cdot\xi_{\nr}, \xi_{\anti}  \}$, and this gives the existence of $x_\cycl,x_\anti$ as in the statement. 

From \eqref{eq: image of xcyc e xant} and Lemma \ref{lemma: 2} we deduce that $\bar{\pi}_*(\bar{x}_\cycl)= n\cdot\bar{\xi}_{\cycl},   \bar{\pi}_*(\bar{x}_\anti)=\bar{\xi}_{\anti} $ and that $\bar{\pi}_*$ is surjective. 
The map $\iota_*$ of \eqref{eq: les} restricts to an isomorphism $\iota_*:\coh^1_f(K_p,V^+)\cong\coh^1_f(K_p,V)$, hence $\bar{\pi}_*$ is also injective.

Finally, in order to understand the action of Frobenius, note that $\coh^1_s(K_p,V)\simeq \coh^1_s(K_p,V^-)$ is naturally a $\Gal(K_p/\Q_p)$-module.  As explained in the proof of Lemma \ref{lemma: 2}, $\{\xi_{\cycl},\xi_{\anti}\}$  is a basis of
\begin{equation}\label{eq: bla}
\coh^1_s(K_p,V^-)\cong  \hom(\Gamma_{\cycl},\Q_p(\chi_E))\oplus \hom(\Gamma_{\anti},\Q_p(\chi_E)).
\end{equation} 

We have
\begin{equation*}
\xi_{\anti}^{\Frob_p}(\sigma_{\anti})=\chi_{E}(\Frob_p)\cdot\xi_{\anti}(\Frob_p\sigma_{\anti}\Frob_p^\inv)=a\cdot\xi_{\anti}(\sigma_{\anti}^\inv)=-a\cdot\xi_{\anti}(\sigma_{\anti})
\end{equation*}
and via  \eqref{eq: bla} it follows that 
\begin{equation*}
\coh^1_s(K_p,V^-)^{-a}\cong\hom(\Gamma_{\anti},\Q_p(\chi_E))
\end{equation*}
is generated by $\bar{\xi}_{\anti}$.
Analogously, $\bar{\xi}_{\cycl}$ generates $	\coh^1_s(K_p,V^-)^{a}\cong\hom(\Gamma_{\cycl},\Q_p(\chi_E)) $.

\end{proof}

\subsection{Bloch--Kato logarithm and dual exponential maps}


We wish to describe explicitly the Bloch--Kato logarithm and the dual exponential maps relative to the representation $V = V_f$ regarded as a $G_{K_p}$-module. First we need to study $\Dderham(V_{{K_p}})$, $\Dderham(V_{{K_p}}^\pm)$ and their filtration, where $V_{{K_p}}:=V_{|G_{K_p}}$.
Let $\C_p$ be the completion of a fixed algebraic closure of $\Q_p$, and let
\begin{equation*} 
\cR:=\varprojlim_{x\mapsto x^p}\cO_{\C_p}/p\cO_{\C_p}.
\end{equation*}

Let $W(\cR)$ be the ring of Witt vectors of $\cR$ and denote $[\cdot]:\cR\longrightarrow W(\cR)$ the Teichmuller lift.
There is an isomorphism
\begin{equation*}
\underset{x\mapsto x^p}{\varprojlim}\cO_{\C_p}\overset{\cong}{\longrightarrow}\cR, \qquad (x^{(n)})_n\mapsto (x_n:= x^{(n)} \mod p)_n. 
\end{equation*}

Recall the basis $\tilde{\eps}:=(\tilde{\eps}^{(m)})_m$ of $\Z_p(1)$ we introduced in \S\ref{sec: tate module of E and group of connected component}. It can be regarded as an element of $\lim_{x\mapsto x^p}\cO_{\C_p}\cong\cR$, and if $\log_{q_E}$ denotes the branch of the $p$-adic logarithm such that $\log_{q_E}(q_E)=0$, then
\begin{equation*}
t:=\log_{q_E}([\tilde{\eps}])
\end{equation*}
generates the maximal ideal $\fil^1\BB_\derham$ of $\BB_\derham^+=\fil^0\BB_\derham$. Analogously, we can regard $\tilde{q}$ as an element of $\cR$ and we define
\begin{equation*}
q_E^\#:=\log_{q_E}([\tilde{q}]/q_E)\in\fil^1\BB_\derham
\end{equation*}
Recall that the images $\eps,q$ of $\tilde{\eps}, \tilde{q}$ under Tate's uniformization form a basis of $V$. By elementary calculations one obtains the following lemma.

\begin{lemma}\label{lem: n-1}
	Let 
	\begin{equation*}
	e_1:=q\otimes 1 - \eps\otimes\dfrac{q_E^\#}{t} \in V\otimes\fil^{0}\BB_\derham, \quad e_2:=\eps\otimes\dfrac{1}{t} \quad \in V\otimes\fil^{-1}\BB_\derham. 
	\end{equation*}
	These elements are invariant with respect to the action of $G_{K_p}$, they form a $K_p$-basis of $\Dderham(V_{K_p})$ and  $e_1$ is a $K_p$-basis of $\fil^0\Dderham(V_{{K_p}})$. Moreover, there are  isomorphisms
	\begin{equation*}\label{eq: iso fil0Vf and fil0Vf+}
	\iota_*:\Dderham(V_{{K_p}}^+)/\fil^0\Dderham(V_{{K_p}}^+)\overset{\cong}{\longrightarrow}\Dderham(V_{{K_p}})/\fil^0\Dderham(V_{{K_p}}), \ \quad \  \eps\otimes t^{-1}\mapsto e_2;
	\end{equation*}
	\begin{equation*}\label{eq: iso fil0Vf and fil0Vf-}
	\pi_*:\fil^0\Dderham(V_{{K_p}})\overset{\cong}{\longrightarrow}\fil^0\Dderham(V_{{K_p}}^-), \ \qquad \  e_1\mapsto \bar{q},
	\end{equation*}
	where, with a slight abuse of notation,  $\bar{q}:=\bar{q}\otimes 1$ in $(V^-\otimes \BB_\derham^+)^{G_{K_p}}=	\fil^0\Dderham(V^-)$.
\end{lemma}

 Kummer theory provides isomorphisms
 \begin{equation*}\label{eq: classical kummer with H1f}
 	\delta_p:\cO_{K_p}^\times\otimes\Q_p\overset{\cong}{\longrightarrow}\coh^1_f(K_p,\Q_p(1)), \quad E(K_p)\otimes\Q_p \overset{\cong}{\longrightarrow}\coh^1_f(K_p,V),
 \end{equation*}
 which combined with \eqref{eq: explicit Tf+, Tf-} and Tate's uniformization 
 yield canonical identifications of $\Q_p[G_{\Q_p}]$-modules
\begin{equation} \label{prop: H1f Kp, E}
 \cO_{K_p}^\times(\chi_E)\otimes\Q_p \, \cong \, \coh^1_f(K_p,V^+) \, \overset{\iota_*}{\cong} \, \coh^1_f(K_p,V) =  E(K_p)\otimes\Q_p.
 \end{equation}

Besides, the $p$-adic logarithm induces a morphism of $\Q_p[G_{\Q_p}]$-modules
\begin{equation*}\label{eq: log bk OK}
\log_{p}:\cO_{K_p}^\times(\chi_E)\otimes\Q_p\longrightarrow K_p(\chi_E).
\end{equation*} 
Via the identifications \eqref{prop: H1f Kp, E}, it gives rise to the map
\begin{align}\label{eq: log bk E}
\log_{E}\colon E(K_p)\longrightarrow K_p(\chi_E),
\end{align}
which is the formal group logarithm on $E$ associated with the invariant differential $\omega_f$ on $E$ introduced in  \S\ref{sec: filtrations}.
Finally, \eqref{prop: H1f Kp, E} and the isomorphism
\begin{equation*}
	\Dderham(V_{K_p})/\fil^0\Dderham(V_{K_p})\overset{\cong}{\longrightarrow}K_p(\chi_E), \qquad e_2\mapsto 1,
\end{equation*} allow us to identify $\log_{E}$ with the composition of the Bloch--Kato logarithm
\begin{equation*}\label{eq: log bk Vf}
	\log:\coh^1_f(K_p,V)\longrightarrow\Dderham(V_{K_p})/\fil^0\Dderham(V_{K_p})
\end{equation*}  with the pairing with $\omega_f$.

For a $p$-adic de Rham representation $W$ of $G_{K_p}$, let
\begin{equation*}
\gamma: \fil^0\Dderham(W)=(W\otimes_{\Q_p}\BB_\derham^+)^{G_{K_p}}\longrightarrow\coh^1(K_p,W\otimes_{\Q_p}\BB_\derham^+)
\end{equation*}
be the isomorphism defined by $x\mapsto\gamma(x)$, where $\gamma(x)$ is the cohomology class of the map
\begin{equation*}\label{eq: formula for gamma}
\sigma\mapsto x\cdot\log_p(\chi_{\cycl}(\sigma)).
\end{equation*}
\begin{definition}\label{BKexp}
	The \textit{Bloch--Kato dual exponential} of $W$ is the map
	\begin{equation*}
	\exp^*:\coh^1(K_p,W)\overset{\alpha_*}{\longrightarrow}\coh^1(K_p,W\otimes_{\Q_p}\BB_\derham^+)\overset{\gamma^{-1}}{\longrightarrow}\fil^0\Dderham(W).
	\end{equation*}
\end{definition}

The map
$\coh^1(K_p,W\otimes \BB_\derham^+)\longrightarrow\coh^1(K_p,W\otimes \BB_\derham)$ is injective and thus $\ker(\exp^*)=\coh^1_g(K_p,W)$. We shall regard the dual exponential associated to $W$ as the map
\begin{equation*}\label{eq: dual exponential in general}
\exp^*:\coh^1_s(K_p,W) = \coh^1(K_p,W)/\coh^1_g(K_p,W) \longrightarrow\fil^0\Dderham(W).
\end{equation*}

Similarly as in the case of the Block--Kato logarithm, the dual exponential
 for $W=K_p$ together with the identification $\coh^1_s(K_p,\Q_p)(\chi_E)=\coh^1_s(K_p,\Q_p(\chi_E))$ gives rise to an isomorphism of $K_p[G_{\Q_p}]$-modules
\begin{equation*}\label{eq: exp per Qp(chiE)}
	\exp^*:\coh^1_s(K_p,\Q_p(\chi_E))\longrightarrow K_p(\chi_E).
\end{equation*} 
It coincides with the dual exponential 
\begin{equation}\label{eq: exp per Vf}
	\exp^*:\coh^1_s(K_p,V)\longrightarrow	K_p(\chi_E)
\end{equation} via  the identifications
\begin{equation*}
	\coh^1_s(K_p,\Q_p(\chi_E))\cong\coh^1_s(K_p,V^-)\cong\coh^1_s(K_p,V)
\end{equation*}
explained in the previous section, where 
\begin{equation*}
	\fil^0\Dderham(V_{{K_p}})\overset{\cong}{\longrightarrow} K_p(\chi_E), \qquad e_1\mapsto 1.
\end{equation*} 



It follows directly from the very definitions and  the $G_{\Q_p}$-equivariance of the dual exponential map that
\begin{equation*}\label{defz}
\exp^*(\xi_{\cycl}) = 1 \,\, \mbox{ and } \,\, z:= \exp^*(\xi_{\anti}) \in K_p^{-}.
\end{equation*}
The map \eqref{eq: exp per Vf} and its $\pm$-components can be explicitly described as follows.

\begin{proposition}\label{prop: exp of Vf}
The dual exponential $\exp^*:\coh^1_s(K_p,V)\longrightarrow K_p(\chi_E)$ is characterized by
	\begin{equation*}
\exp^*(\bar{x}_\cycl)=n, \ \  \exp^*(\bar{x}_\anti)=z \in K_p^-.
	\end{equation*}

\end{proposition}
\begin{proof}

	By (the proof of) Proposition \ref{prop:1},
	\begin{equation*}
		\coh^1_s(K_p,\Q_p(\chi_E))^{a}\oplus\coh^1_s(K_p,\Q_p(\chi_E))^{-a}\cong \hom(\Gamma_{\cycl},\Q_p(\chi_E))\oplus \hom(\Gamma_{\anti},\Q_p(\chi_E)),
	\end{equation*} 
whose basis $\{\xi_{\cycl},\xi_{\anti}\}$ is compatible with the above decomposition. 

	The claim follows by using Lemma \ref{lem: n-1}, Proposition \ref{prop:1} and the description of \eqref{eq: exp per Vf} discussed before this proposition.
\end{proof}


\begin{corollary}\label{cor: main appendix}
	There is a commutative diagram whose arrows are isomorphisms of $G_{\Q_p}$-modules: 
	\begin{equation*}
		\xymatrix{\hom_{\cont}(\Gamma_{\cycl},\Phi_{\infty,\fp})\oplus\hom_{\cont}(\Gamma_{\anti}, \Phi_{\infty,\fp})\ar[r]^{\qquad\qquad\qquad\varphi}\ar[d]_{\frac{1}{p}\mathrm{ev}_\cycl}^{\oplus\frac{1}{p}\mathrm{ev}_\anti}& \coh^1_s(K_p,V)\ar[d]^{\exp^*}\\
		\left(	\Phi_{\infty,\fp}\otimes\Q_p\right)\oplus	\left(	\Phi_{\infty,\fp}\otimes\Q_p\right)\ar[r]^{\qquad\qquad\varrho} & K_p(\chi_E)			
		}
	\end{equation*}
	where
\begin{enumerate}
	\item $\varphi:=\left(\bar{\varphi}_\Tate\circ\bar{\pi}_*\right)^{-1}$;
	\item $\varrho(\bar{q},0)=1, \ \varrho(0,\bar{q})=z$, where $z= \exp^*(\xi_{\anti}) \in K_p^{-}$ is as in Proposition \ref{prop: exp of Vf};
	\item $\mathrm{ev}_{\cycl}$ denotes the valuation at $\sigma_{\cycl}\in\Gamma_{\cycl}$, and analogously for $\mathrm{ev}_{\anti}$.
\end{enumerate}
Moreover, $\varphi$ decomposes with respect to the action of $\Frob_p$ as
\begin{equation}\label{eq: varphi pm}
	 \hom_{\cont}(\Gamma_{\cycl},\Phi_{\infty,\fp})\longrightarrow\coh^1_s(K_p,V)^{a}, \qquad  \hom_{\cont}(\Gamma_{\anti},\Phi_{\infty,\fp})\longrightarrow\coh^1_s(K_p,V)^{-a}.
\end{equation}
	
\end{corollary}
\begin{proof}
	By definition we have
	\begin{align*}
	\xi_{\cycl}(\sigma_{\cycl})  =  \overline{\varphi}_\Tate(\log_p(\chi_{\cycl}(\sigma_{\cycl})))
	 = \overline{\varphi}_\Tate(\log_p(u))
	 = \overline{\varphi}_\Tate(p)
	 = p\cdot \bar{q}
	\end{align*}
	and similarly, $\xi_{\anti}(\sigma_{\anti}) = p\cdot\bar{q}$.
The statement follows by combining (the proof of) Proposition \ref{prop: exp of Vf}, Proposition \ref{prop:1} and Lemma \ref{lem: Phi infty and limit of bar Phi1 + Phi infty and limit of bar Phi2 + rem: 2}.
\end{proof}

\color{black}

\section{Hida families and Selmer groups associated to $(f, g, h)$}\label{sec: the selmer group of fxgxh}

\subsection{Filtrations and differentials}\label{sec: filtrations}

Let \begin{equation*}
	\varphi\in S_w(M,\eps)
\end{equation*} be an ordinary newform of weight $w$, level $M$ and character $\eps$ with coefficients in a finite extension $L/\Q$. Fix an algebraic closure $\bar\Q$ of $\Q$ and an embedding $L \subset \bar\Q$. Fix also a prime number $p$, an algebraic closure $\bar\Q_p$ of $\Q_p$ and an embedding $\bar\Q \hookrightarrow \bar\Q_p$, with respect to which we assume  $\varphi$ is ordinary. Let $L_p \subset \bar\Q_p$ denote the $p$-adic closure of $L$.

Let $V_\varphi$ denote the $p$-adic Galois representation attached to $\varphi$ with coefficients in $L_p$ and let $\alpha_\varphi,\beta_\varphi$ be the roots of the Hecke polynomial $x^2+a_p(\varphi)x+\varepsilon(p)$ at $p$ of $V_\varphi$. Since $\varphi$ is ordinary, one of these roots, say $\alpha_{{\varphi}}$, is a $p$-adic unit.  As a $G_{\Q_p}$-representation, $V_\varphi$ is equipped with a decreasing filtration of $L_p[G_{\Q_p}]$-modules
\begin{equation}\label{eq: filtration Vphi}
\fil^2(V_\varphi)=0\subseteq\fil^1(V_\varphi)\subseteq V_\varphi=\fil^0(V_\varphi).
\end{equation}
Let \begin{equation*}
\psi_\varphi: G_{\Q_p}\longrightarrow L_p^\times
\end{equation*} be the unramified character such that $\psi_\varphi(\Frob_p)=\alpha_\varphi$.
The filtration \eqref{eq: filtration Vphi} satisfies:
\begin{enumerate}
	\item[(i)] $V_\varphi^+:=\fil^1(V_\varphi)$ and $V_\varphi^-:=V_\varphi/\fil^1(V_\varphi)$ have dimension one over $L_p$;
	\item[(ii)]  $G_{\Q_p}$ acts on $V_\varphi^+$ via $\chi_{\cycl}^{w-1}\eps\psi_\varphi^{-1}$;
	\item[(iii)] $G_{\Q_p}$ acts on $V_\varphi^-$ via $\psi_\varphi$. In particular, it is an unramified $G_{\Q_p}$-module.
\end{enumerate}


When $w\geq2$ Faltings' comparison theorem yields an isomorphism
 \begin{equation}\label{eq: iso Sw fil1derham} 
 S_w(M)_{L_p}[\varphi]\overset{\cong}{\longrightarrow} \fil^1\Dderham(V_\varphi)\cong \Dderham(V_\varphi^-),
 \end{equation} 
 where  $S_w(M)_{L_p}[\varphi]$ denotes the $\varphi$-isotipical component of $S_w(M)\otimes L_p$.
We define
\begin{equation*} 
	\omega_\varphi\in \Dderham(V_\varphi^-)
\end{equation*}  
to be the image of $\varphi$ via \eqref{eq: iso Sw fil1derham}.

Denote $\varphi^*:=\varphi\otimes\eps^{-1}\in S_w(M,\bar{\eps})$.  Consider the pairing
\begin{equation*}\label{eq: pairing Vphi, Vphi*}
\langle \ , \ \rangle: V_{\varphi}\times V_{\varphi^*}(1)\longrightarrow L_p
\end{equation*}
 given by Poincar\'e's duality. This in turn  induces perfect pairings
\begin{equation*}\label{eq: pairing DV*, DV} 
	\langle \ , \ \rangle: \Dderham(V_{\varphi})\times \Dderham(V_\varphi^*(1))\longrightarrow \Dderham(L_p)=L_p
\end{equation*}
and
\begin{equation*}
\langle \ , \ \rangle: \Dderham(V_\varphi^+) \times S_w(M)_{L_p}[\varphi^*] \longrightarrow L_p.
\end{equation*}
Define
\begin{equation*}
	\eta_\varphi\in \Dderham(V_\varphi^+)\subseteq\Dderham(V_\varphi)
\end{equation*}
to be the differential characterized by 
\begin{equation*}
\langle \omega_{\varphi^*}, \eta_\varphi\rangle=1.
\end{equation*}


If $\varphi$ has weight $w=1$, then we will always assume that $p\nmid M$. In this case, both $\alpha_\varphi$ and $\beta_\varphi$ are $p$-adic units, and  throughout this paper we will also assume that 
\begin{equation*}
	\alpha_{{\varphi}}\neq\beta_\varphi.
\end{equation*} 
In this setting, the filtration \eqref{eq: filtration Vphi} admits a splitting. More precisely, if we 
denote $V_\varphi^\alpha, V_\varphi^\beta$ the eigenspaces for the action of $\Frob_p$ on $V_\varphi$, with eigenvalues $\alpha_\varphi, \beta_\varphi$ respectively, then 
\begin{equation*}
V_\varphi=V_\varphi^\alpha\oplus V_\varphi^\beta, \quad V_\varphi^+=V_\varphi^\beta, \qquad V_\varphi^-=V_\varphi^\alpha
\end{equation*}
as $G_{\Q_p}$-modules.
In this case, there is a pairing 
\begin{equation}\label{pairing weight one} 
	\Dderham(V_\varphi)\times\Dderham(V_\varphi)\longrightarrow \Dderham(L_p(\chi_\varphi)),
\end{equation} and we let
\begin{equation}\label{differentials}
	\omega_{\varphi_\alpha}\in\Dderham(V_\varphi^\alpha), \  \ \eta_{\varphi_\alpha}\in\Dderham(V_\varphi^\beta)
\end{equation}
be the differentials attached by Ohta to the $p$-stabilized eigenform $\varphi_\alpha$ as described in \cite[\S2]{DR2b}. 
They satisfy
\begin{equation}\label{eq: pairing between omega and eta}
	\langle \omega_{\varphi_\alpha},\eta_{\varphi_\alpha}\rangle =  \mathfrak{g}(\varepsilon) \in L_p,
\end{equation}
where $\mathfrak{g}(\varepsilon)$ denotes the Gauss sum of the nebentype character $\varepsilon$ of $\varphi$.

Attached to a weight $1$ modular form $\varphi$ there is an Artin representation
\begin{equation*}
\rho_\varphi:G_\Q\rightarrow\GL(V^{\circ}_\varphi)
\end{equation*} 
where $V^{\circ}_\varphi$ is an $L$-vector space equipped with a (non-canonical) isomorphism $j_\varphi:V^{\circ}_\varphi\otimes_L  L_p\overset{\cong}{\rightarrow} V_\varphi$. The choice of $j_\varphi$ thus induces an $L$-structure on $V_\varphi$ by $V_\varphi^L:=j_\varphi(V^{\circ}_\varphi)$. 

Let  $v_\varphi^{\alpha}$ be an $L$-basis of $V_\varphi^L\cap V_\varphi^\alpha$ and $v_\varphi^\beta$ an $L$-basis of $V_\varphi^L\cap V_\varphi^\beta$. Let $H$ be the number field cut out by $\rho_\varphi$, fix once and for all an embedding $H\subseteq\bar{\Q}_p$ and denote $H_p$ the $p$-adic completion of $H$ inside $\bar{\Q}_p$.  Finally, define periods
$$\Omega_{\varphi_\alpha}\in H_p^{\Frob_p=1/\alpha_\varphi}, \ \Theta_{\varphi_\alpha}\in H_p^{\Frob_p=1/\beta_\varphi}$$ characterized as 
\begin{equation}\label{eq: definition of Omega and Theta}
\Omega_{\varphi_\alpha}\otimes v_\varphi^\alpha=\omega_{\varphi_\alpha}\in \Dderham(V_g^\alpha), \ \ \Theta_{\varphi_\alpha}\otimes v_\varphi^\beta=\eta_{\varphi_\alpha}\in \Dderham(V_\varphi^\beta).  
\end{equation}
Although the above periods depend on the choice of $j_\varphi$, the ratio
\begin{equation}\label{Lga}
\L_{g_\alpha}:=\dfrac{\Omega_{g_\alpha}}{\Theta_{g_\alpha}}
\end{equation}
is readily seen to be independent of this choice; it only depends on the choice of basis $\{ v_\varphi^{\alpha}, v_\varphi^{\beta} \}$ and is thus well defined up to scalars in $L^\times$.

\vspace{0.3cm}

Let $\Lambda:=\Z_p[[1+p\Z_p]]$ be the Iwasawa algebra and 
$\cW:=\mathrm{Spf}(\Lambda)$ be the weight space. It is the rigid analytic space whose $A$-points for any $\Z_p$-algebra $A$ are 
$$\cW(A)=\hom_{\Z_p-\mathrm{alg}}(\Lambda,A)=\hom_\mathrm{cts}(1+p\Z_p,A^\times).$$ Recall that $\Z$ can be seen as a dense subset of $\cW(\C_p)$ identifying an integer $k\in\Z$ with the character $\nu_k:(x\mapsto x^k)\in\hom_\mathrm{cts}(1+p\Z_p,\C_p^\times)$. More generally, given a Dirichlet character $\epsilon$ of $p$-power conductor, we denote $\nu_{k,\epsilon}$ the point $(x\mapsto\epsilon(x)x^k)\in\cW(\C_p)$. A character $\nu \in \cW$ is called  \textit{classical} if it is of the form $\nu=\nu_{k,\epsilon}$ with $k\geq 2$ and $k$ is called its \textit{weight}. A classical point  is \textit{cristalline} if it is of the form $\nu_{k,\omega^k}$, where $\omega$ is the Teichm\"uller character.

Let  $$\boldsymbol{\varphi}=\sum a_n({\boldsymbol{\varphi}})q^n\in\Lambda_{\boldsymbol{\varphi}}[[q]]$$ be a Hida family of tame level $M$ and tame character $\chi$, where $\Lambda_{\boldsymbol{\varphi}}$ is a finite flat extension of $\Lambda$, and let $\kappa:\cW_{\boldsymbol{\varphi}}\to\cW$ be the \textit{weight map} induced by the inclusion $\Lambda\subseteq\Lambda_{\boldsymbol{\varphi}}$. A point $x\in\cW_{\boldsymbol{\varphi}}$ is called \textit{classical} (resp.\,\textit{cristalline}) if $\kappa(x)$ is, and we denote $\cW_{\boldsymbol{\varphi}}^{\mathrm{cl}}$ (resp. $\cW_{\boldsymbol{\varphi}}^{\circ}$) the set of classical (resp.\,cristalline) points of $\cW_{\boldsymbol{\varphi}}$. For each $x\in\cW_{\boldsymbol{\varphi}}^{\mathrm{cl}}$ with $\kappa(x)=\nu_{k,\epsilon}$, the specialization
\begin{equation*}
	\boldsymbol{\varphi}_x:=\sum_{n\geq1}x(a_n(\boldsymbol{\varphi}))q^n
\end{equation*}
of $\boldsymbol{\varphi}$  at $x$
is the $q$-expansion of a classical $p$-ordinary eigenform of weight $k$ and character $\chi\epsilon\omega^{-k}$. If $x\in \cW$ is a point of weight $1$, meaning that $\kappa(x)=\nu_{1,\epsilon}$ for some Dirichlet character $\epsilon$ of $p$-power conductor, the specialization $\boldsymbol{\varphi}_x$ may be a classical modular form or not. (Later we shall regard the $p$-stabilizations of $g$ and $h$ as weight $1$ specializations of a Hida family, and these are classical by construction.)

Notice that the specializations of $\boldsymbol{\varphi}$ at cristalline points have all nebentype $\chi$. If $x\in\cW_{\boldsymbol{\varphi}}^{\circ}$ is a cristalline point of weight $k>2$, then $\boldsymbol{\varphi}_x$ is the ordinary $p$-stabilization of a newform $\varphi_x\in S_k(N,\chi)$. If $k=2$, then $\boldsymbol{\varphi}_x$ may be either old or new at $p$. We denote $\varphi_x:=\boldsymbol{\varphi}_x$ if it is new, while if $\boldsymbol{\varphi}_x$ is old at $p$, we denote $\varphi_x$ the newform whose $p$-stabilization is $\boldsymbol{\varphi}_x$.
 If no confusion is likely to arise, we may often denote these forms  by an abuse of notation as $\boldsymbol{\varphi}_k$ and $\varphi_k$.

Let $V_{\boldsymbol{\varphi}}$ be the Galois representation of $G_{\Q}$  attached to ${\boldsymbol{\varphi}}$, a $\Lambda_{\boldsymbol{\varphi}}$-module of rank $2$.
As a $G_{\Q_p}$-representation there is again a filtration
\begin{equation*}\label{eq: filtration VPhi}
\fil^2(V_{\boldsymbol{\varphi}})=0\subseteq \fil^1(V_{\boldsymbol{\varphi}})\subseteq V_{\boldsymbol{\varphi}}=\fil^0(V_{\boldsymbol{\varphi}})
\end{equation*}
where $V_{\boldsymbol{\varphi}}^+:=\fil^1(V_{\boldsymbol{\varphi}})$ has dimension one, and the quotient  $V_{\boldsymbol{\varphi}}^-:=V_{\boldsymbol{\varphi}}/\fil^1(V_{\boldsymbol{\varphi}})$ is unramified. More precisely, $G_{\Q_p}$ acts on $V_{\boldsymbol{\varphi}}^-$ via the unramified character
\begin{equation*}\label{eq: psiPhi}
\psi_{\boldsymbol{\varphi}}:G_{\Q_p}\longrightarrow\Lambda_{\boldsymbol{\varphi}}^\times
\end{equation*}
such that $\psi_{\boldsymbol{\varphi}}(\Frob_p)=a_p({\boldsymbol{\varphi}})$. 

%
%

If $x$ is cristalline and $L_p/\Q_p$ is a finite extension containing the values of $x$, then  the Galois representation $V_{\boldsymbol{\varphi}}\otimes_{\Lambda_{\boldsymbol{\varphi}},x}L_p\cong V_{\varphi_x}$ is cristalline at $p$.

Fix now three cuspidal Hida families
\begin{equation*}
	\bff\in\Lambda_{\bf f}[[q]], \qquad \bfg\in\Lambda_{\bfg}[[q]], \qquad \bfh\in\Lambda_{\bfh}[[q]].
\end{equation*}

Write $\Lambda_{\mathbf{fgh}}  = \Lambda_{\bff} \hat\otimes \Lambda_{\bfg} \hat\otimes \Lambda_{\bfh}$ and $\cO_{\mathbf{fgh}}:=\Lambda_{\bff}[1/p]\hat{\otimes}\Lambda_{\bfg}[1/p]\hat{\otimes}\Lambda_{\bfh}[1/p]$. Let $\cW_{\mathbf{fgh}} = \mathrm{Spf}(\Lambda_{\mathbf{fgh}}) = \cW_\bff\times\cW_\bfg\times\cW_\bfh$ denote the associated three-dimensional product of weight spaces. Denote $\cW_{\mathbf{fgh}}^{\circ}$ the set of cristalline  points $(x,y,z)\in\cW_\bff^{\circ}\times\cW_\bfg^{\circ}\times\cW_\bfh^{\circ}$ of weights $(k,\ell,m)\in\Z_{\geq2}\times\Z_{\geq1} \times \Z_{\geq1}$  such that $k\equiv 2\mod(p-1), \ell\equiv m\equiv1\mod2(p-1)$. As above, by abuse of notation we 
may identify points $(x,y,z)$ of $\cW_{\mathbf{fgh}}^{\circ}$ with their weights $(k,\ell,m)$.
  
Set $$\V:= V_\bff\otimes V_\bfg\otimes V_\bfh\otimes_{\cO_{\mathbf{fgh}}}\Xi,$$ 
where $\Xi:G_{\Q_p}\longrightarrow\cO_{\mathbf{fgh}}^\times$ is the character defined in \cite[\S4.6.2]{BSV19reciprocity}
  and is characterized by the property that, if  $\underline{x}=(k,\ell,m)\in\cW_{\mathbf{fgh}}^{\circ}$ is a  triple of cristalline points  and $\rho_{\underline{x}}:=\nu_k\otimes\nu_\ell\otimes\nu_m:\cO_{\mathbf{fgh}}^\times\longrightarrow L_p$ is the corresponding specialisation map, then $\rho_{\underline{x}}\circ\Xi=\chi_{\cycl}^{(4-k-\ell-m)/2}$. We define a filtration on $\V$ by
\begin{equation*}
	\fil^q(\V):=[\oplus_{i+j+t=q}\fil^i(V_\bff)\hat{\otimes} \fil^j(V_\bfg)\hat{\otimes} \fil^t(V_\bfh)]\otimes_{\cO_{\mathbf{fgh}}}\Xi,
\end{equation*}

Since $\V/\fil^2(\V)$ has no $G_{\Q_p}$-invariants, the natural inclusion $$\fil^2(\V)= [V_\bff\otimes V_\bfg^+\otimes V_\bfh^+\otimes_{\cO_{\mathbf{fgh}}}\Xi]\oplus  [V_\bff^+\otimes V_\bfg\otimes V_\bfh^+\otimes_{\cO_{\mathbf{fgh}}}\Xi]\oplus [V_\bff^+\otimes V_\bfg^+\otimes V_\bfh\otimes_{\cO_{\mathbf{fgh}}}\Xi] \hookrightarrow \V$$
 induces an injective morphism in cohomology
\begin{equation*}
	\coh^1(\Q_p,\fil^2(\V))\longrightarrow \coh^1(\Q_p,\V),
\end{equation*} 
and henceforth we shall identify $\coh^1(\Q_p,\fil^2(\V))$ with its image in $\coh^1(\Q_p,\V)$.
Defining
\begin{equation}\label{VfVgVh}
\V^f:=V_\bff^-\otimes V_\bfg^+\otimes V_\bfh^+\otimes_{\cO_{\mathbf{fgh}}}\Xi, \quad \V^g:=V_\bff^+\otimes V_\bfg^-\otimes V_\bfh^+\otimes_{\cO_{\mathbf{fgh}}}\Xi, \quad \V^h:=V_\bff^+\otimes V_\bfg^+\otimes V_\bfh^-\otimes_{\cO_{\mathbf{fgh}}}\Xi,
\end{equation} it follows that $\coh^1(\Q_p,\fil^2(\V)/\fil^3(\V))$ decomposes as
\begin{equation*}
\coh^1(\Q_p,\V^f)\oplus\coh^1(\Q_p,\V^g)\oplus \coh^1(\Q_p,\V^h).
\end{equation*}

For $\varphi\in\{ f,g,h \}$, we denote 
\begin{equation*}
	\pi_\varphi: \coh^1(\Q_p,\fil^2(\V))\longrightarrow\coh^1(\Q_p,\V^\varphi)
\end{equation*}
the natural projection map.

\subsection{Selmer groups}\label{sec: selmer groups}
 Let $W$ be a $L_p[G_\Q]$-module. Given a prime number $\ell$, define as costumary
\begin{equation*}\label{eq: selmer structure}
\coh_f^1(\Q_\ell,W):=\begin{cases}  \coh^1_\mathrm{ur}(\Q_\ell,W):=\coh^1(\Q_\ell^\mathrm{ur}/\Q_\ell,W^{I_\ell}) & \ell\neq p\\ \ker\Big(\coh^1(\Q_p,W)\rightarrow\coh^1(\Q_p,W\otimes_{\Q_p} \BB_\mathrm{cris})   \Big) & \ell=p,	\end{cases}
\end{equation*}
and 
\begin{equation*}
\coh^1_s(\Q_\ell,W):=\coh^1(\Q_\ell,W)/\coh^1_f(\Q_\ell,W).
\end{equation*}
The Bloch--Kato Selmer group of $W$ is $$\Sel_p(W):=\{x\in\coh^1(\Q,W)\mid \res_\ell(x)\in\coh_f^1(\Q_\ell,W) \text{ for all }\ell \},$$
where $\res_\ell\colon \coh^1(\Q,W)\ra \coh^1(\Q_\ell,W)$ denotes the restriction map in Galois cohomology.

For each prime $\ell$, we denote by $\partial_\ell$ the composition 
\begin{equation*}
\partial_\ell: \coh^1(\Q,W)\overset{\res_\ell}{\longrightarrow}\coh^1(\Q_\ell,W)\longrightarrow \coh^1_s(\Q_\ell,W).
\end{equation*}

The relaxed Selmer group is defined as $$\Sel_{(p)}(W):=\{ x\in\coh^1(\Q,W)\mid \res_\ell(x)\in\coh_f^1(\Q_\ell,W) \ \text{ for all }\   \ell\neq p \}\supseteq\Sel_p(\Q,W).$$

\vspace{0.3cm}

Let now $(f,g,h)$ be the triple of eigenforms of weights $(2,1,1)$ introduced at the beginning of the article.
Recall Assumption \ref{ass: rank 0} imposed on $V:=V_f\otimes V_g\otimes V_h$ in the introduction, which implies 
\begin{equation*}
\Sel_p(V)=0.
\end{equation*}
Indeed, as discussed in detail in \S\ref{sec: Decomposition of the representation $V$ and consequences}, the representation $V$ decomposes as a direct sum $$V=\left[V_f\otimes V_{\psi_1}\right]\oplus \left[V_f\otimes V_{\psi_2}\right],$$ where, for $i=1,2$, the $G_\Q$-representation $V_{\psi_i}$ is induced by a ring class character $\psi_i$ of the imaginary quadratic field $K$. This decomposition induces factorisations 
\begin{equation*}
	L(E,\rho,s)= L(E/K,\psi_1,s)\cdot L(E/K,\psi_2,s), \qquad \Sel_p(V)=\Sel_p(E\otimes\psi_1)\oplus\Sel_p(E\otimes\psi_2).
\end{equation*}
Combining the factorisation of the complex $L$-function with Assumption \ref{ass: rank 0} we obtain that $L(E/K,\psi_i,1)\neq0$ for $i=1,2$, which, by \cite[Theorem B]{BD97}, implies that $\Sel_p(E\otimes\psi_1)=\Sel_p(E\otimes\psi_2)=0$. 


\begin{lemma}\label{lem: relaxed selmer iso with H1s}
There is an isomorphism
	\begin{equation*}
	\partial_p: \Sel_{(p)}(V)\overset{\cong}{\longrightarrow} \coh^1_s(\Q_p,V)
	\end{equation*}
\end{lemma}
\begin{proof}
This follows from Poitou--Tate duality as in 
 \cite[Theorem 2.3.4]{MR04}, where a  similar statement is proved, under the (irrelevant) assumption that $p\nmid N_f$.
\end{proof}

The characteristic polynomial for the action of $\Frob_p$ on $V_g$ is $x^2-a_p(g)x+\chi(p)=x^2-\eps(p)$ and we may thus write the eigenvalues of $\Frob_p$ as  $\alpha_g=\lambda$, $\beta_g=-\lambda$ for some root of unity $\lambda$. The same holds for $h$ and since its nebentype is the inverse of that of $g$ we have
\begin{equation*}
	(	\alpha_h,\beta_h)=(1/\lambda,-1/\lambda) \ \text{ or } \  (	\alpha_h,\beta_h)=(-1/\lambda,1/\lambda).
\end{equation*}\color{black}
Using the notation of \S\ref{sec: filtrations}, denote $\{v_g^\alpha, v_g^\beta \}$ and  $\{v_h^\alpha, v_h^\beta\}$ a pair of  $L$-bases for $V_g$ and $V_h$ consisting of eigenvectors with eigenvalue $\alpha_g, \beta_g, \alpha_h, \beta_h$ respectively. Denote $V_g^\alpha, V_g^\beta$, etc. the corresponding eigenspaces. Set $V_{gh}^{\alpha \alpha} := V_g^\alpha \otimes V_h^\alpha$ and $V^{\alpha\alpha}:=V_f\otimes V_{gh}^{\alpha \alpha}$, and likewise for the remaining pairs of eigenvalues.
There is a decomposition of $L_p[G_{\Q_p}]$-modules
\begin{equation}\label{eq: decomposition in eigenspaces}
	V:= V_{fgh} = V^{\alpha\alpha}\oplus V^{\alpha\beta}\oplus V^{\beta\alpha}\oplus V^{\beta\beta}.
\end{equation}


\begin{lemma}\label{lem: tipo lem 4.1 of DR3}
	Let $\triangle,\heartsuit\in\{\alpha,\beta\}$. Then
		\begin{enumerate}[$i)$]
			\item the Bloch--Kato logarithm gives an isomorphism 	\begin{equation*}
			\log_{\triangle\heartsuit}:\coh_f^1(\Q_p,V^{\triangle\heartsuit})=\coh^1_f(\Q_p,V_f^+\otimes V_{gh}^{\triangle\heartsuit})\overset{\cong}{\longrightarrow} L_p;
			\end{equation*}
			\item the Bloch--Kato dual exponential gives an isomorphism 	\begin{equation*}
			\exp^*_{\triangle\heartsuit}:\coh_s^1(\Q_p,V^{\triangle\heartsuit})=\coh^1_s(\Q_p,V_f^-\otimes V_{gh}^{\triangle\heartsuit})\overset{\cong}{\longrightarrow} L_p;.
			\end{equation*}
		\end{enumerate}

\end{lemma}

\begin{proof}
	Let $W:=V_{gh}^{\triangle\heartsuit}$. By \cite[Lemma 2.4.1]{DR19Stark}
		\begin{equation}\label{eq: cases for H1f}
			\coh^1_f(\Q_p,V^{\triangle\heartsuit})=\begin{cases}
			\coh^1_f(\Q_p,V_f^+\otimes V_{gh}^{\triangle\heartsuit})	& \text{ if } \triangle\cdot\heartsuit\cdot a=+1\\
			\coh^1(\Q_p,V_f^+\otimes V_{gh}^{\triangle\heartsuit}) & \text{ if } \triangle\cdot\heartsuit\cdot a=-1.
			\end{cases}
		\end{equation}
		Moreover, in the case in which $\triangle\cdot\heartsuit\cdot a=-1$, we have 
		\begin{equation*}
			\coh^1(\Q_p,V_f^+\otimes V_{gh}^{\triangle\heartsuit})=\coh^1_f(\Q_p,V_f^+\otimes V_{gh}^{\triangle\heartsuit})
		\end{equation*}
		(see \cite[Example 3.1.4]{DR19padic}). 
		Since $V_f^+\cong\Q_p(1)(\chi_E)$, $ i)$ follows from the fact that 
		\begin{equation*}
			V_f^+\otimes V_{gh}^{\triangle\heartsuit}=L_p(\zeta\chi_{\cycl})
		\end{equation*}
		where $\zeta:G_{\Q_p}\longrightarrow\Q_p^\times$ is the unramified character given by $\zeta(\Frob_p)=\triangle\cdot\heartsuit\cdot a$ (for more details, see e.g. \cite[Example 3.1.4]{DR19padic}).  The statement on the singular quotients is proven similarly using \eqref{eq: cases for H1f} and the fact that 
		\begin{equation*}
		V_f^-\otimes V_{gh}^{\triangle\heartsuit}=L_p(\zeta).
		\end{equation*}
\end{proof}

Set $L_p^4= L_p\oplus L_p\oplus L_p\oplus L_p$. We obtain from the previous lemma isomorphisms
 	\begin{align*}
 	\log:\coh^1_f(\Q_p,V)\longrightarrow L_p^4, \qquad
 	\exp^*:\coh^1_s(\Q_p,V)\longrightarrow L_p^4.
 	\end{align*}
 	where
 	$\log=\log_{\alpha\alpha}\oplus\log_{\alpha\beta}\oplus\log_{\beta\alpha}\oplus\log_{\beta\beta}$ and 	$\exp^*=\exp^*_{\alpha\alpha}\oplus\exp^*_{\alpha\beta}\oplus\exp^*_{\beta\alpha}\oplus\exp^*_{\beta\beta}$.
 	
 Combining Lemma \ref{lem: relaxed selmer iso with H1s} with Lemma \ref{lem: tipo lem 4.1 of DR3} we obtain the following result on the $L_p$-structure of the relaxed Selmer group attached to $V$.
 	
 \begin{corollary}\label{cor: Sel(p) has dimension 4} 
 	There are isomorphisms
 	\begin{equation*}
 		\Sel_{(p)}(V)\overset{\partial_p}{\longrightarrow}\coh^1_s(\Q_p,V)\overset{\exp^*}{\longrightarrow}L_p^4.
 	\end{equation*} 	
 \end{corollary}

We now turn to pin down a particular basis of $\Sel_{(p)}(V)$ which is canonical up to multiplication by elements in $L^\times$. In order to do this, we need a more precise description of the finite and singular parts of $\coh^1(\Q_p,V^{\heartsuit\triangle})$ in terms of $\coh^1(K_p,V_f)$. Note that the latter space is equipped with a natural action of $\Gal(K_p/\Q_p)$.  

In what follows, for any $\Gal(K_p/\Q_p)$-module $M$ we denote $M^{\pm}$ the subspace of $M$ on which $\Frob_p$ acts as multiplication by $\pm1$. Note that $\alpha_g\alpha_h, \alpha_g\beta_h, \beta_g \alpha_h, \beta_g\beta_h \in \{ \pm 1\}$ and thus it makes sense to consider $M^{\alpha_g\alpha_h}$, etc.


\begin{lemma}\label{lem: H1(Q,V) and H(K,Vf)}
	Let $\heartsuit,\triangle\in\{\alpha,\beta\}$.
	There are canonical isomorphisms of $L_p$-vector spaces
	\begin{equation*}
		\coh^1_f(\Q_p,V_f\otimes V_{gh}^{\heartsuit\triangle}) \cong \coh^1_f(K_p,V_f)^{\triangle\heartsuit}\otimes V_{gh}^{\heartsuit\triangle} 
	\end{equation*}
	and
	\begin{equation*}
	\coh^1_s(\Q_p,V_f\otimes V_{gh}^{\heartsuit\triangle}) \cong \coh^1_s(K_p,V_f)^{\triangle\heartsuit} \otimes V_{gh}^{\heartsuit\triangle}.
	\end{equation*}
\end{lemma}
\begin{proof}
		Let $\chi$ be the quadratic character of $G_{\Q_p}$ with $\chi(\Frob_p)=-1$. We describe the  isomorphisms in the case in which $E$ has split multiplicative reduction at $p$ and $\triangle\cdot\heartsuit=-1$. In this setting,
\begin{align*}
\coh^1_f(\Q_p,V^{\triangle\heartsuit}) & = 	\coh^1_f(\Q_p,V_f^+\otimes V_{gh}^{\triangle\heartsuit}) & \text{ by Lemma \ref{lem: tipo lem 4.1 of DR3}}\\
			& = \coh^1_f(\Q_p,\Q_p(1)\otimes L_p(\chi)) &  \text{by Lemma \ref{eq: filtration Vphi} (ii)}  \\
			& = \coh^1_f(K_p,\Q_p(1)\otimes L_p(\chi))^{G_{\Q_p}} & \text{ by the inflation-restriction exact sequence} \\
			& = (\coh^1_f(K_p,\Q_p(1))\otimes L_p(\chi))^{G_{\Q_p}} & \text{ because $\chi_{|G_{K_p}}=1$} \\
			& = (\coh_f^1(K_p,V_f)\otimes L_p(\chi))^{G_{\Q_p}} &    \text{ by \eqref{prop: H1f Kp, E}.}
\end{align*}
The claim follows, because $V_{gh}^{\triangle\heartsuit}=L_p(\chi)$ as $G_{\Q_p}$-modules and the subspace of $G_{\Q_p}$-invariants of $\coh_f^1(K_p,V_f)\otimes V_{gh}^{\triangle\heartsuit}$ is $\coh^1_f(K_p,V_f)^{\triangle\heartsuit}\otimes V_{gh}^{\heartsuit\triangle}$.
Using similar computations as above, one also verifies that
\begin{equation*}
	\coh^1_s(\Q_p,V^{\triangle\heartsuit}) \cong(\coh^1(K_p,V_f)\otimes V_{gh}^{\triangle\heartsuit})^{G_{\Q_p}}/(\coh^1_f(K_p,V_f)\otimes V_{gh}^{\triangle\heartsuit})^{G_{\Q_p}} \cong(\coh^1_s(K_p,V_f)\otimes V_{gh}^{\triangle\heartsuit})^{G_{\Q_p}}.
\end{equation*}
The remaining cases are proven similarly. 
\end{proof}

\begin{remark}\label{rem: logaritmos}  
The logarithm maps of Lemma \ref{lem: tipo lem 4.1 of DR3} are twisted versions of \eqref{eq: log bk E}. 
More precisely, $\log_{\alpha\beta}$ may be accordingly recast  as

\begin{align}\label{logab}
\begin{array}{ccccc}
H^1_f(\Q_p,V_f\otimes  V_{gh}^{\alpha\beta}) = & E(K_p)^{\beta\alpha}\otimes V_{gh}^{\alpha\beta} & \longrightarrow &  K_p(\chi_E)^{\beta\alpha}\otimes K_p^{\alpha\beta}\otimes L_p\cong L_p\\
& x\otimes v_g^\alpha v_h^\beta & \longmapsto & \log_{E}(x)\cdot \langle v_g^\alpha v_h^\beta , \eta_{g_\alpha}\omega_{h_\alpha}\rangle.
\end{array}
\end{align} 
Using Proposition \ref{prop: exp of Vf} we can define the following generators $k^+$ and $k^-$ of $K_p(\chi_E)^+$ and $K_p(\chi_E)^-$ respectively:  \begin{equation*}
	k^+:=\begin{cases}
	1 & \text{ if } a=+1;\\ z & \text{ if }a=-1
	\end{cases} \ \text{ and } \ 	k^-:=\begin{cases}
	z & \text{ if }a=+1;\\ 1 & \text{ if } a=-1. 
	\end{cases}
	\end{equation*}
Using \eqref{eq: pairing between omega and eta} and \eqref{eq: definition of Omega and Theta}, the latter pairing is 
\begin{equation*}
\langle v_g^\alpha v_h^\beta , \eta_{g_\alpha}\omega_{h_\alpha}\rangle  = \frac{1}{\Omega_{g_\alpha}\Theta_{h_\alpha}}\langle \omega_{g_\alpha}\eta_{h_\alpha},\eta_{g_\alpha}\omega_{h_\alpha}\rangle = \frac{1}{\Omega_{g_\alpha}\Theta_{h_\alpha}}\otimes1\in K_p^{\alpha\beta}\otimes L_p,
\end{equation*}
and the right-most isomorphism of \eqref{logab} is given by 
\begin{align*}
\begin{array}{ccccc}
K_p(\chi_E)^{\beta\alpha}\otimes K_p^{\alpha\beta}& \overset{\cong}{\longrightarrow} & K_p(\chi_E)^+ &\overset{\cong}{\longrightarrow}& \Q_p\\
x\otimes y & \longmapsto &xy & & \\
& & k^{+} & \longmapsto & 1.
\end{array}
\end{align*}

Analogously, the dual exponential maps of Lemma \ref{lem: tipo lem 4.1 of DR3} are twisted versions of \eqref{eq: exp per Vf}, and $\exp^*_{\alpha\beta}$ may be recast  under the identification provided by Lemma \ref{lem: H1(Q,V) and H(K,Vf)} as
\begin{align*}
\begin{array}{ccccc}
\exp^*_{\alpha\beta}:  \coh^1_s(K_p,V_f)^{\beta\alpha} \otimes V_{gh}^{\alpha\beta} & \longrightarrow &  L_p\\
x\otimes v_g^\alpha v_h^\beta & \longmapsto &  \exp^*(x)\cdot\langle v_g^\alpha v_h^\beta , \eta_{g_\alpha}\omega_{h_\alpha}\rangle.
\end{array}
\end{align*}

\end{remark}
We will still denote by 
\begin{equation}\label{eq: exp pm onto Qp}
\exp_\pm^*:\coh_s^1(K_p,V_f)^\pm\longrightarrow\Q_p
\end{equation} the composition of \eqref{eq: exp per Vf} with the isomorphism $K_p(\chi_E)^\pm\cong\Q_p$ sending $k^\pm$ to $1$.

\begin{corollary}\label{cor: basis of Sel(p)}
	 Define 
	 $X_\pm\in \coh^1_s(K_p,V_f)^{\pm}$ such that $\exp^*_\pm(X_\pm)=1$. 
The $p$-relaxed Selmer group $\Sel_{(p)}(V)$ admits a basis
\begin{equation*}
	\{ {\xi}^{\alpha\alpha}, {\xi}^{\alpha\beta}, {\xi}^{\beta\alpha}, {\xi}^{\beta\beta} \}
\end{equation*}	
 characterized as
\begin{align*}
	\partial_p{\xi}^{\alpha\alpha}=X_{{\alpha\alpha}}\otimes v_g^\alpha\otimes v_h^\alpha, & \qquad \partial_p{\xi}^{\alpha\beta}=X_{{\alpha\beta}}\otimes v_g^\alpha\otimes v_h^\beta, \\ \partial_p{\xi}^{\beta\alpha}=X_{{\beta\alpha}}\otimes v_g^\beta\otimes v_h^\alpha, & \qquad	\partial_p{\xi}^{\beta\beta}=X_{{\beta\beta}} \otimes v_g^\beta\otimes v_h^\beta.
\end{align*}
\end{corollary}
\begin{proof}
By Corollary \ref{cor: Sel(p) has dimension 4} and using the decomposition \eqref{eq: decomposition in eigenspaces}, there are isomorphisms
\begin{equation}\label{eq: x}
	\Sel_{(p)}(V) \overset{\partial_p}{\longrightarrow} \coh^1_s(\Q_p,V) \cong \coh^1_s(\Q_p,V^{\alpha\alpha})\oplus\coh^1_s(\Q_p,V^{\alpha\beta})\oplus\coh^1_s(\Q_p,V^{\beta\alpha})\oplus\coh^1_s(\Q_p,V^{\beta\beta}).
\end{equation}
By Lemma \ref{lem: H1(Q,V) and H(K,Vf)},  each of the four components in \eqref{eq: x} is isomorphic to 
\begin{align*}
 \coh^1_s(K_p,V_f)^{{\beta\beta}}\otimes V_{gh}^{\alpha\alpha}, \,  \coh^1_s(K_p,V_f)^{{\beta\alpha}}\otimes V_{gh}^{\alpha\beta}, \,  \coh^1_s(K_p,V_f)^{{\alpha\beta}}\otimes V_{gh}^{\beta\alpha}, \, \coh^1_s(K_p,V_f)^{{\alpha\alpha}}\otimes V_{gh}^{\beta\beta},
\end{align*}
respectively. \ These are in turn isomorphic to $V_{gh}^{\alpha\alpha}\oplus V_{gh}^{\alpha\beta}\oplus V_{gh}^{\beta\alpha}\oplus V_{gh}^{\beta\beta}$ via \ 
  $\exp^*_{{\alpha\alpha}}\oplus\exp^*_{{\alpha\beta}}\oplus\exp^*_{{\beta\alpha}}\oplus\exp^*_{{\beta\beta}}$. 
\end{proof}

\section{A special value formula for the triple product $p$-adic $L$-function in rank $0$}\label{sec: main thm 1}

	The aim of this section is to describe the $p$-adic $L$-value $I_p(f,g_\alpha,h_\alpha)$ in terms of the basis $\{ \xi^{\alpha\alpha}, \xi^{\alpha\beta}, \xi^{\beta\alpha}, \xi^{\beta\beta} \}$ of $\Sel_{(p)}(V)$ appearing in Corollary \ref{cor: basis of Sel(p)}. This section lies within the framework of the \textit{exceptional setting} of \cite{BSV19reciprocity} and we recall here the notation and the main results from loc.\,cit.\,that we shall use.

	\subsection{The triple product $p$-adic $L$-function}

	If $\varphi=\sum a_n(\varphi) q^n$ is a modular form  of weight $w$, level $M$ and nebentype character $\chi_{\varphi}$ and $p\nmid M$, the Hecke polynomial at $p$ of $\varphi$ is 
	$$x^2-a_p(\varphi)x+\chi_\varphi(p)p^{w-1}=(x-\alpha_\varphi)(x-\beta_\varphi),$$
	where we label the eigenvalues so that $\ord_p(\alpha_\varphi)\leq\ord_p(\beta_\varphi)$. Recall that if $\varphi$ is ordinary at $p$, then $\alpha_\varphi$ is a $p$-adic unit.
	If $p\mid M$ we have $\alpha_\varphi = a_p(\varphi)$. 
	
	Since $g$ has weight $w=1$,  both $\alpha_g$ and $\beta_g$ are $p$-adic units. Recall that we are assuming that $\alpha_g\neq\beta_g$, thus $g$ has two different ordinary $p$-stabilisations. We denote $g_\alpha$ and $g_\beta$ the stabilisation satisfying $U_pg_\alpha=\alpha_g g_\alpha$ and $U_pg_\beta=\beta_g g_\beta$ respectively. The same holds for $h$, and we denote similarly its $p$-stabilisations.
	Since $E$ has multiplicative reduction at $p$, the level of the newform is divisible by $p$ and it is $p$-stabilised, meaning that $f_\alpha=f$.

	Let $\bff=\sum_{n\geq1}a_n(\bff)\in\Lambda_{\bff}[[q]], \ \bfg\in\Lambda_\bfg[[q]], \ \bfh\in\Lambda_\bfh[[q]]$ be the Hida families passing through $f, g_\alpha, h_\alpha$ respectively, where $\Lambda_\bff,\Lambda_\bfg$ and $\Lambda_\bfh$ are finite flat extensions of the Iwasawa algebra $\Lambda$. These Hida families have tame level $N_f/p, N_g, N_h$ and tame character $1, \chi, \bar{\chi}$ respectively.

	Let $N:=$lcm$(N_f/p,N_g,N_h)$. A \textit{test vector} for $(\bff,\bfg,\bfh)$ is a triple of Hida families $(\ubff,\ubfg,\ubfh)$ of tame level $N$ such that $\ubff$ is of the form $\sum\lambda_d {\bf f}(q^d)\in \Lambda_{\bf f}[[q]]$ with $\lambda_d\in \Lambda_{\bf f}$, where $d$ runs over the divisors of $N/N_f$, and similarly for $\ubfg$ and $\ubfh$.

	Recall from \S\ref{sec: filtrations} the set $\cW_{\mathbf{fgh}}^{\circ}$ of cristalline points. Denote  $\cW_{\mathbf{fgh}}^{f}:=\{ (k,\ell,m)\in\cW_{\mathbf{fgh}}^{\circ}\mid k\geq\ell+m \}$ and  define $\cW_{\mathbf{fgh}}^{g}$ and $\cW_{\mathbf{fgh}}^{h}$ analogously. Triplets of points in $\cW_{\mathbf{fgh}}^{f}\cup\cW_{\mathbf{fgh}}^{g}\cup\cW_{\mathbf{fgh}}^{h}$ are called \textit{unbalanced}, and points in the complement $\cW_{\mathbf{fgh}}^{\bal}:=\cW_{\mathbf{fgh}}^{\circ}\smallsetminus(\cW_{\mathbf{fgh}}^{f}\cup\cW_{\mathbf{fgh}}^{g}\cup\cW_{\mathbf{fgh}}^{h})$ are called \textit{balanced}.
	
	Let 
	\begin{equation*}
	\L_p^f(\ubff,\ubfg,\ubfh), \L_p^g(\ubff,\ubfg,\ubfh), \L_p^h(\ubff,\ubfg,\ubfh):\cW_\bff\times\cW_\bfg\times\cW_\bfh\longrightarrow \C_p
	\end{equation*}
	be the triple product $p$-adic $L$-functions attached to $(\ubff,\ubfg,\ubfh)$ constructed in \cite{DR1}. 

	\subsection{Exceptional cases and improved Euler systems}\label{sec: exceptional cases}
	
	As introduced in \eqref{sec: filtrations}, let $$\V = V_\bff\otimes V_\bfg\otimes V_\bfh\otimes_{\cO_{\mathbf{fgh}}}\Xi$$ and let
	\begin{equation*}
	\boldsymbol{\kappa}:=\boldsymbol{\kappa}(\bff,\bfg,\bfh)\in\coh^1(\Q, \V)
	\end{equation*}
	be the diagonal class constructed in  \cite{BSV19reciprocity} and \cite{DR19padic}.  
	By  \cite[Corollary 4.7.1]{BSV19reciprocity} and \cite[Proposition 3.5.7]{DR19padic}, the local class
	\begin{equation*}
	\boldsymbol{\kappa}_p:=\res_p(\boldsymbol{\kappa})\in\coh^1(\Q_p, \V)
	\end{equation*}
	belongs to $\coh^1(\Q_p,\fil^2(\V))$. Moreover, for $\varphi\in\{ f,g,h \}$, the $p$-adic $L$-function $\L_p^\varphi(\ubff,\ubfg,\ubfh)$ can be recast as the image of $\boldsymbol{\kappa}_p$ under the idoneous Perrin-Riou's $\Lambda$-adic logarithm: more precisely,   \cite[Theorem A]{BSV19reciprocity}, \cite[Theorem 2.29]{DR19padic} assert that
	\begin{equation}\label{eq: thmA bsv}
	\L_p^\varphi(\ubff,\ubfg,\ubfh)=\L_\varphi(\res_p(\boldsymbol{\kappa})),
	\end{equation}
	where $\L_\varphi:\coh^1(\Q_p,\fil^2(\V))\longrightarrow\cO_{\mathbf{fgh}}$ is the homomorphism described in \cite[Proposition 4.6.2]{BSV19reciprocity}, \cite[Proposition 3.5.6]{DR19padic}.
	
	Let $\cH$ denote the surface cut out by the equation $k=2+\ell-m$ in $\cW_\bff\times\cW_\bfg\times\cW_\bfh$ and let $\mathcal{F}_\cH$ be the fraction field of the ring of Iwasawa functions on $\cH$.

	Define the two-variable meromorphic Iwasawa function
	\begin{equation*}
	\boldsymbol{\cE}_g:=	\cE_g(\bff,\bfg,\bfh):=1-\dfrac{a_p(\bfg_\ell)}{\chi(p)a_p(\bff_{\ell-m+2})a_p(\bfh_m)}\in \mathcal{F}_\cH.
	\end{equation*}
	
	
	Recall from the introduction that we have $p\mid\mid N_f$, $p\nmid N_g N_h$, and 	$\alpha_g\alpha_h=-a.$
	More explicitly, we have \begin{equation}\label{alfa e betas}
		\beta_g = -\alpha_g,  \qquad (\alpha_h,\beta_h) = \begin{cases}
		(-1/\alpha_g,1/\alpha_g) & \text{ if } a=+1;\\ (1/\alpha_g,-1/\alpha_g) & \text{ if } a=-1.
		\end{cases}
	\end{equation}
	Hence
	\begin{equation*}\label{eq: exceptional condition Eg} 
	\boldsymbol{\cE}_g(2,1,1)=1-\dfrac{\alpha_g\beta_h}{a}=0.  
	\end{equation*}
	This fact forces the vanishing of the class $\boldsymbol{\kappa}(2,1,1)$, as explained in \cite{BSV19reciprocity} (cf.\,also Prop. \ref{prop: definition of improved class} below).

%
%
%
%

 Let
 $\rho_{\cH}:\mathcal{F}_{\mathbf{fgh}}\longrightarrow\mathcal{F}_{\cH}$  be the map taking a function $F(k,\ell,m)$ to its restriction $F(2+\ell-m,\ell,1)$ to the plane $\cH$.  Denote
  $\V_{|\cH}:=\V\otimes_{\mathcal{F}_{\mathbf{fgh}},\rho_{\cH}}\mathcal{F}_{\cH}$ and let $$\boldsymbol{\kappa}_{|\cH} := \rho_{\cH, \star}(\boldsymbol{\kappa})\in\coh^1(\Q,\V_{|\cH})$$ denote the restriction of $\boldsymbol{\kappa}$ to $\cH$.
 
		
	\begin{proposition}\label{prop: definition of improved class}
		There exists a global cohomology class
		\begin{equation*}
		\boldsymbol{\kappa}_g^*\in\coh^1(\Q,\V_{|\cH})
		\end{equation*}
		satisfying, for each $\ell,m \in \Z_{\geq 1}$:
		\begin{equation}\label{eq: kappa and kappa improved}
		\boldsymbol{\kappa}(2+\ell-m,\ell,m)=\boldsymbol{\cE}_g(2+\ell-m,\ell,m)\cdot\boldsymbol{\kappa}_g^*(2+\ell-m,\ell,m).
		\end{equation}
	\end{proposition}
	\begin{proof}
		This is shown in \cite[\S8.3]{BSV19reciprocity}
	\end{proof}

Let $\cC \subset \cH$ denote the curve given on which the set of points $(\ell+1,\ell,1)$ for $\ell \in \Z_{\geq 1}$ is dense. 

	\begin{proposition}\label{prop: Lpf* and kappa*g}
	There exists an analytic function $\L_p^{f,*}$ on $\cC$ satisfying, for each $\ell \in \Z_{\geq 1}$:
		\begin{equation*}
		\L_p^{f,*}(\ell)=\dfrac{\langle w_N(f_{\ell+1})^{(p)}, h\bfg_\ell\rangle}{\langle w_N(f_{\ell+1})^{(p)},w_N(f_{\ell+1})^{(p)}\rangle}
		\end{equation*}
		and 
		\begin{equation*}\label{Lpf*} 
		\L_p^f(\ubff,\ubfg,\ubfh)(\ell+1,\ell,1)=\left(1-\dfrac{a_p(\bfg_\ell)\alpha_h}{a_p(\bff_{\ell+1})}\right)\boldsymbol{\cE}_g(\ell+1,\ell,1)\L_p^{f,*}(\ell) \mod L^\times.
		\end{equation*}
		Moreover, 
		\begin{equation}\label{eq: (170) of BSV} 
		\L_p^{f,*}(1)=\dfrac{1}{2(1-1/p)} \exp^*_{\beta\beta}(\pi_{\beta\beta}\partial_p\boldsymbol{\kappa}_g^*(2,1,1))=\dfrac{1}{2(1-1/p)} \exp^*_{\beta\beta}(\partial_p\boldsymbol{\kappa}_g^*(2,1,1)).
		\end{equation}
	\end{proposition}
	\begin{proof}
	This is proved in \cite[Lemma 8.6 + equation (170)]{BSV19reciprocity}. In particular, combining Theorem A, part 3 of Proposition 8.2 and Lemma 8.6 of loc.\ cit.\ it follows that the factor $\lambda_{w_0}$ in equation (170) is $1/2(1-1/p)$. 
	\end{proof}
	\begin{remark}
		By \eqref{alfa e betas} the factor $1-\dfrac{a_p(\bfg_\ell)\alpha_h}{a_p(\bff_{\ell+1})}$  does not vanish in a neighborhood of $(2,1,1)$ in $\cC$.
	\end{remark}

	\subsection{$I_p(f,g_\alpha,h_\alpha)$ in terms of the basis}\label{sec: Lpg in therms of the basis}
	In this section we finally obtain, in Theorem \ref{thm: main1}, the formula for $I_p(f,g_\alpha,h_\alpha)$.
	Recall from the introduction the Galois representation
	$$
	V=V_f\otimes V_g \otimes V_h.
	$$
	Note that  the character $\Xi$ introduced in \S \ref{sec: filtrations} specializes at $(2,1,1)$ to the trivial character, and thus the specialization of the improved $\Lambda$-adic cohomology class of Prop.\,\ref{prop: definition of improved class} yields a global cohomology class
	\begin{equation*}
	\boldsymbol{\kappa}^*_g(2,1,1)\in\Sel_{(p)}(V).
	\end{equation*}
	

	Recall the periods  
	\begin{equation}\label{perio}
	\Omega_{g_\alpha}\in K_p^{1/\alpha_g}, \quad \Theta_{g_\alpha}\in K_p^{1/\beta_g},  \quad \L_{g_\alpha}:=\dfrac{\Omega_{g_\alpha}}{\Theta_{g_\alpha}} \in K_p^{\beta_g/\alpha_g}
	\end{equation}
	 introduced in \S\ref{sec: filtrations}.

	\begin{proposition}\label{prop: kappa* and the basis} We have
		$$\boldsymbol{\kappa}_g^*(2,1,1)=\Theta_{g_\alpha}\Theta_{h_\alpha}\dfrac{2(1-1/p)\sqrt{c}\sqrt{L(E\otimes\rho,1)}}{ \pi^2\langle f, f\rangle} \, \cdot \, \xi^{\beta\beta},$$
		where $c\in L^\times$ is the product of the local terms appearing in \cite[Proposition 2.1 (iii)]{DLR}.
	\end{proposition}
	\begin{proof}
		According to \cite[\S8.3]{BSV19reciprocity} one has $\res_p(\boldsymbol{\kappa}_g^*(2,1,1))\in\coh^1(\Q_p,\fil^2(V))$. Recall that 
		\begin{equation*}
		\coh^1(\Q_p,\fil^2(V)/\fil^3(V)) = \coh^1(\Q_p, V_f^-\otimes V_{gh}^{\beta\beta})\oplus \coh^1(\Q_p, V_f^+\otimes V_{gh}^{\alpha\beta})\oplus \coh^1(\Q_p, V_f^+\otimes V_{gh}^{\beta\alpha})
		\end{equation*}  
		and hence 
		 $\partial_p\boldsymbol{\kappa}_g^*(2,1,1)\in\coh_s^1(\Q_p,V)=\coh_s^1(\Q_p,V_f^-\otimes V_{gh})$ lies in $\coh^1_s(\Q_p,V^{\beta\beta})$.
		Then, by definition of the basis of $\Sel_{(p)}(V)$ described in Corollary \ref{cor: basis of Sel(p)},
		\begin{equation*}
		\boldsymbol{\kappa}_g^*(2,1,1) = \dfrac{\exp^*_{\beta\beta}(\partial_p\boldsymbol{\kappa}_g^*(2,1,1))}{\exp^*_{\beta\beta}(\partial_p\xi^{\beta\beta})}\cdot \xi^{\beta\beta}.
		\end{equation*}
		We need to compute the numerator and denominator of the previous formula.
		By \eqref{eq: (170) of BSV}, 
		\begin{align*}
		\exp^*_{\beta\beta}(\partial_p\boldsymbol{\kappa}_g^*(2,1,1)) & =2(1-1/p) \L_p^{f*}(1)\\
		& =\dfrac{2(1-1/p)\langle w_N(f), hg\rangle}{\langle w_N(f), w_N(f)\rangle}\\
		& =\dfrac{2(1-1/p)\sqrt{c}\sqrt{L(E\otimes\rho,1)}}{\pi^2\langle f, f\rangle},
		\end{align*}
		where the last equality is given by \cite{Ich08}. Besides,
		\begin{align*}
		\exp_{\beta\beta}^*(\partial_p\xi^{\beta\beta}) & = \langle \exp^*_{\beta\beta}\partial_p\xi^{\beta\beta}, \eta_f\omega_{g_\alpha}\omega_{h_\alpha}\rangle \\
		& =   {\exp^* (X_{{\beta\beta}})}\langle v_g^\beta\otimes v_h^\beta, \omega_{g_\alpha}\otimes\omega_{h_\alpha}\rangle\\
		& = \dfrac{1}{\Theta_{g_\alpha}\Theta_{h_\alpha}} \langle  \eta_{g_\alpha}\otimes\eta_{h_\alpha}, \omega_{g_\alpha}\otimes\omega_{h_\alpha}\rangle\\
		& = \dfrac{1}{\Theta_{g_\alpha}\Theta_{h_\alpha}}.
		\end{align*}
		The proposition follows.
	\end{proof}

	Let
	$\Lambda_{\cycl} = \Z_p[[j+1]]$ denote the usual Iwasawa algebra regarded as the ring of bounded analytic functions on an open disc centered at $j=-1$, and let  
	$\boldsymbol{\chi}_{\cycl}:G_{\Q}\longrightarrow \Lambda_{\cycl}^\times$  denote the $\Lambda$-adic cyclotomic character characterized by the property that $\nu_j(\boldsymbol{\chi}_{\cycl})=\chi_{\cycl}^{-j}$ for all $j\in\Z$. 
	
	Recall the three-variable Iwasawa algebra $\Lambda_{\mathbf{fgh}}$ and set $$\bar{\Lambda}_{\mathbf{fgh}} = \Lambda_{\mathbf{fgh}} \hat\otimes \Lambda_{\cycl}, \quad \overline{\cW}_{\mathbf{fgh}}:=  \mathrm{Spf}(\bar{\Lambda}_{\mathbf{fgh}}) = \cW_{\mathbf{fgh}}\times \cW, \quad \overline{\cW}^{\circ}_{\mathbf{fgh}}:=\cW_{\mathbf{fgh}}^{\circ}\times\cW^{\circ}.$$ 
	

	Let $\Psi:G_{\Q_p}\longrightarrow\Lambda_{\mathbf{fgh}}^\times$ denote the unramified character taking $\Frob_p$ to $\dfrac{a_p(\bfg)}{a_p(\bff)a_p(\bfh)\chi(p)}$. 
	Define 
	\begin{equation*}
	\mathbb{M}^g:=\Lambda_{\mathbf{fgh}}(\Psi), \qquad \overline{\mathbb{M}}^g:=\mathbb{M}^g\otimes\Lambda_{\cycl}(\boldsymbol{\chi}_{\cycl})
	\end{equation*}
	and note that $\mathbb{M}^g$ is the unramified twist of the local Galois representation $\V^g$ introduced in \eqref{VfVgVh}. 	
		
	Let
	$\theta: \, \bar{\Lambda}_{\mathbf{fgh}}  \, \longrightarrow \, \Lambda_{\mathbf{fgh}}$ denote the homomorphism taking a function $F(k,\ell,m,j)$ to its restriction $F(k,\ell,m,(\ell-k-m)/2)$.  As shown in the proof of \cite[Prop.\,4.6.2]{BSV19reciprocity},
there is an isomorphism
	\begin{equation}\label{eq: theta e eps}
	\theta: \overline{\mathbb{M}}^g\otimes_\theta\cO_{\mathbf{fgh}}\cong\V^g.
	\end{equation}
	Define also the $\Lambda$-adic Dieudonn\'e modules
	\begin{equation*}
	\mathbb{D}^g:=(\mathbb{M}^g\hat{\otimes}\hat{\Z}_p^{\mathrm{ur}})^{G_{\Q_p}}, \quad \overline{\mathbb{D}}^g:=\mathbb{D}^g\otimes\Lambda_{\cycl},
	\end{equation*}
 where $\hat{\Z}_p^{\mathrm{ur}}$ is the ring of integers of the $p$-adic completion $\hat{\Q}_p^{\mathrm{ur}}$ of the maximal unramified extension of $\Q_p$. 	  
	
As it directly follows from the above definitions, the specialization of $\overline{\mathbb{M}}^g$ at a point $\underline{x}=(k,\ell,m,j)\in\overline{\cW}^{\circ}_{\mathbf{fgh}}$ is 
	\begin{equation*}\label{eq: ciao}
	\overline{\mathbb{M}}^g_{\underline{x}} := \overline{\mathbb{M}}^g\otimes_{\underline{x}}L_p = L_p(\Psi_{(k,\ell,m)})(-j)=V_{\bff_k}^+\otimes V_{\bfg_\ell}^-\otimes V_{\bfh_m}^+(2-k-m-j).
	\end{equation*}
%
\color{black}
Bloch--Kato's logarithm and dual exponential maps give rise to an isomorphism
	\begin{align}\label{rem: biglog}
	\log_{\underline{x}}: \coh^1(\Q_p,L_p(\Psi_{(k,\ell,m)})(-j))\longrightarrow\Dderham(L_p(\Psi_{(k,\ell,m)})(-j)) \qquad \text{ if $j<0$}\\
	\exp^*_{\underline{x}}: \coh^1(\Q_p,L_p(\Psi_{(k,\ell,m)})(-j))\longrightarrow\Dderham(L_p(\Psi_{(k,\ell,m)})(-j)) \qquad \text{ if $j\geq0$}.
	\end{align}
	In particular, at $\underline{x}=(2,1,1,-1)$ this isomorphism becomes the map
	\begin{equation*}
	\log_{\alpha\beta}:\coh^1_f(\Q_p,V^{\alpha\beta})\longrightarrow L_p,
	\end{equation*}
	after taking the pairing with the differential  $\eta_{g_\alpha}\omega_{h_\alpha}$ (cf.\,Remark \ref{rem: logaritmos}).
	
	\begin{proposition}\label{prop: interpolation for big log}
		There exists a single homomorphism of $\bar{\Lambda}_{\mathbf{fgh}}$-modules
	\begin{equation*}
	\bar{\L}_g:\coh^1(\Q_p,\overline{\mathbb{M}}^g)\longrightarrow \overline{\mathbb{D}}^g,
	\end{equation*}
	such that for all $\mathbf{Z}\in\coh^1(\Q_p,\overline{\mathbb{M}}^g)$ and $\underline{x}=(k,\ell,m,j)=(w,j)\in\overline{\cW}^{\circ}_{\mathbf{fgh}}$ with \begin{equation*}
		1-\Psi_{w}(\Frob_p) p^{-j-1}\neq0,
		\end{equation*} we have
		\begin{equation*}\label{eq: interpolation for big log}
		\bar{\L}_g(\mathbf{Z})(\underline{x})=\left(1-\dfrac{p^j}{\Psi_{w}(\Frob_p)}\right)(1-\Psi_w(\Frob_p)p^{-j-1})^{-1}\cdot\begin{cases} 
		\dfrac{(-1)^{j+1}}{(-j-1)!}\log_{\underline{x}}(\mathbf{Z}(\underline{x})) & j<0;\\
		j! \ \exp^*_{\underline{x}}(\mathbf{Z}(\underline{x})) & j\geq 0.
		\end{cases}
		\end{equation*}
		
	\end{proposition}
	\begin{proof}
	The proposition as stated here is \cite[Proposition 4.6.1]{BSV19reciprocity} for $\bfg$, and it is a consequence of \cite{LZ14}.
	\end{proof}

	By the proof of \cite[Proposition 4.6.2]{BSV19reciprocity},  the relation between $\bar{\L}_g$ and the logarithm $\L_g$ that appears in \eqref{eq: thmA bsv} is given, for all $\mathbf{Z}\in\coh^1(\Q_p,\overline{\mathbb{M}}^g)$ by
	\begin{equation}\label{eq: big log and bigbig log}
	\L_g(\theta_*(\mathbf{Z}))(k,\ell,m)=\langle\bar{\L}_g(\mathbf{Z})(k,\ell,m,(\ell-k-m)/2),\omega_{\bff_k}\eta_{\bfg_\ell}\omega_{\bfh_m}\rangle.
	\end{equation}
Here, if $\underline{x}=(k,\ell,m,(\ell-k-m)/2)$, then \begin{equation*}
	\omega_{\bff_k}\eta_{\bfg_\ell}\omega_{\bfh_m}\in\overline{\mathbb{D}}^g\otimes_{\underline{x}}L_p\cong\Dderham(V_{\bff_k}^+\otimes V_{\bfg_\ell}^-\otimes V_{\bfh_m}^+((\ell-k-m+4)/{2}))
\end{equation*}
are the differentials defined as in \cite[(228)]{BSV19reciprocity} and are natural generalizations of the ones introduced in \S\ref{sec: filtrations}.

When $j$ is a fixed integer, note that the Euler-like factors appearing above vary analytically. As it will suffice or our purposes, set $j=-1$ and recall the plane $\cH$ parametrized by points of weights $(2+\ell-m,\ell,m)$ in $\cW_{\mathbf{fgh}}$. By a slight abuse of notation  we  also regard $\cH$  as the plane in $\bar{\cW}_{\mathbf{fgh}}$ parametrizing points of weights $(2+\ell-m,\ell,m,-1)$.

Define the homomorphism
	\begin{equation*}
		\bar{\L}^*_g:\coh^1(\Q_p, \overline{\mathbb{M}}^g_{|\cH})\longrightarrow \overline{\mathbb{D}}^g_{|\cH}
		\end{equation*}
of $\cO_{\cH}$-modules given by
\begin{equation}\label{lem: log*1}
\bar{\L}^*_g = {\boldsymbol{\cE}_g} \times 	\bar{\L}_{g|\cH}.
\end{equation}

In light of the above proposition,
for any $\boldsymbol{Z}\in\coh^1(\Q_p, \overline{\mathbb{M}}^g_{|\cH}) $ and $\ell\geq1$ one then has 
\begin{equation}\label{eq: LLg* formula}
		\bar{\L}^*_g(\boldsymbol{Z})(\ell)= (1-\dfrac{a_p(\bfg_\ell)\beta_h}{a_p(\bff_{\ell+1})p}) \cdot \langle \log_{(\ell+1,\ell,1,-1)}(\mathbf{Z}(\ell+1,\ell,1,-1)),\omega_{\bff_{\ell+1}}\eta_{\bfg_\ell}\omega_{h_\alpha}\rangle.
		\end{equation}

		
		
		


	\begin{proposition}\label{prop: Lpg and kappa*} We have
		$$
		I_p(f,g_\alpha,h_\alpha)={\left(1-\dfrac{1}{p}\right)}\log_{\alpha\beta}(\pi_{\alpha\beta}\boldsymbol{\kappa}_g^*(2,1,1)) \mod L^\times.
		$$
		
	\end{proposition}
\begin{proof}	
		To avoid notational clutter, we continue to denote $\cH$ the plane in $\overline{\cW}_{\mathbf{fgh}}$ containing the points $(2+\ell-m,\ell,m,-1)$ as a dense subset.
		Let   $\widetilde{\boldsymbol{\kappa}}^*_{g,p}\in\coh^1(\Q_p,\overline{\mathbb{M}}^g_{|\cH})$ be any lift of $\res_p(\boldsymbol{\kappa}_g^*)$, whose existence is assured by the fact that the map
	\begin{equation*}
		\theta_*:\coh^1(\Q_p,\overline{\mathbb{M}}^g)\otimes_\theta\cO_{\mathbf{fgh}}\overset{}{\longrightarrow}\coh^1(\Q_p,\V^g)
	\end{equation*}
	induced by \eqref{eq: theta e eps} is an isomorphism (cf.\ the proof of \cite[Proposition  4.6.2]{BSV19reciprocity}).

	Combining Proposition \ref{prop: interpolation for big log}  with \eqref{eq: thmA bsv}, \eqref{eq: kappa and kappa improved}, \eqref{eq: big log and bigbig log}  we deduce
	\begin{align*}
	\L_p^g(\ubff,\ubfg,\ubfh)_{|\cH} & = \L_{g|\cH}(\res_p(\boldsymbol{\kappa}_{|\cH}))  = \boldsymbol{\cE}_{g|\cH}\cdot\L_{g|\cH}(\res_p(\boldsymbol{\kappa}^*_{g}))\\
	& = \boldsymbol{\cE}_{g|\cH}\cdot\langle\bar{\L}_{g|\cH}(\widetilde{\boldsymbol{\kappa}}^*_{g,p}), \omega_{\bff_{\ell+1}}\eta_{\bfg_\ell}\omega_{h_\alpha}\rangle.
	\end{align*}

  By \eqref{lem: log*1}, \eqref{eq: LLg* formula}  and  \eqref{rem: biglog}, the latter quantity is equal to 
  $$
\langle   \bar{\L}^*_g(\widetilde{\boldsymbol{\kappa}}^*_{g,p}), \omega_{\bff_{\ell+1}}\eta_{\bfg_\ell}\omega_{h_\alpha}\rangle = (1-\dfrac{a_p(\bfg_\ell)\beta_h}{a_p(\bff_{\ell+1})p}) \cdot \log_{\alpha\beta}(\pi_{\alpha\beta}\boldsymbol{\kappa}_g^*(\ell+1,\ell,1)).
  $$ 
Since $I_p(f,g_\alpha,h_\alpha) = \L_p^g(\ubff,\ubfg,\ubfh)(2,1,1)$ by definition, the proposition follows by using \eqref{alfa e betas}.
	\end{proof}
	
Define the local point $$P_{{\alpha\beta}}\in E(K_p)^{{\alpha\beta}}_{L_p}$$ as the one satisfying  
	\begin{equation}\label{defPab}
	\pi_{\alpha\beta}\res_p(\xi^{\beta\beta})=\delta_p(P_{{\alpha\beta}})\otimes v_g^\alpha\otimes v_h^\beta\in\left(\coh^1_f(K_p,V_f)\otimes V_{gh}^{\alpha\beta}\right)^{G_{\Q_p}},
	\end{equation}
	where $\delta_p:E(K_p)\longrightarrow\coh^1_f(K_p,V_f)$ is the Kummer map and $E(K_p)^{{\alpha\beta}}_{L_p}:=E(K_p)^{{\alpha\beta}}\otimes{L_p}$.

\begin{theorem}\label{thm: main1}  
We have
	\begin{equation*}  
	I_p(f,g_\alpha, h_\alpha) =\dfrac{2(1-1/p)^2\sqrt{c}}{\pi^2\langle f, f\rangle}\times {\sqrt{L(E\otimes\rho,1)}} \times\dfrac{\log_p(P_{{\alpha\beta}})}{\L_{g_\alpha}} \ \mod L^\times.
	\end{equation*}
\end{theorem}
\begin{proof}
	
	By Proposition \ref{prop: kappa* and the basis}, 
	$\boldsymbol{\kappa}_g^*(2,1,1)=\dfrac{\Theta_{g_\alpha}\Theta_{h_\alpha}2(1-1/p)\sqrt{c}\sqrt{L(E\otimes\rho,1)}}{\pi^2\langle f, f\rangle}\xi^{\beta\beta}$. It thus follows from Proposition \ref{prop: Lpg and kappa*} that
	\begin{align*}
	I_p(f,g_\alpha,h_\alpha) &= {\left(1-\dfrac{1}{p}\right)}  \log_{\alpha\beta}(\pi_{\alpha\beta}\boldsymbol{\kappa}_g^*(2,1,1))\\
	& = \dfrac{\Theta_{g_\alpha}\Theta_{h_\alpha}\sqrt{c}2(1-1/p)^2\sqrt{L(E\otimes\rho,1)}}{\pi^2\langle f, f\rangle} \log_{\alpha\beta}(\pi_{\alpha\beta}\xi^{\beta\beta})\\
	& = \dfrac{\Theta_{g_\alpha}\Theta_{h_\alpha}\sqrt{c}2(1-1/p)^2\sqrt{L(E\otimes\rho,1)}}{\pi^2\langle f, f\rangle} \log_p(P_{{\alpha\beta}})\langle v_g^\alpha\otimes v_h^\beta, \eta_{g_\alpha}\omega_{h_\alpha}\rangle\\
	& = \dfrac{\Theta_{g_\alpha}\Theta_{h_\alpha}\sqrt{c}2(1-1/p)^2\sqrt{L(E\otimes\rho,1)}}{\Omega_{g_\alpha}\Theta_{h_\alpha}\pi^2\langle f, f\rangle} \log_p(P_{{\alpha\beta}}).\\
	\end{align*}
	The theorem follows in light of \eqref{perio}.	
\end{proof}

\section{A special value formula for $I_p(f,g_\alpha,h_\alpha)$ in terms of Kolyvagin classes}\label{sec: main thm 2}

The aim of this section is to relate Theorem \ref{thm: main1} to Bertolini-Darmon's Kolyvagin classes constructed in \cite[\S6]{BD97}. 

\subsection{Decomposition of the representation $V$ and consequences}\label{sec: Decomposition of the representation $V$ and consequences}

The underlying reason why  the $p$-adic $L$-value $\L_p^g(\ubff,\ubfg,\ubfh)(2,1,1)$ may be related to Kolyvagin classes relies on the decomposition of the Galois representation $V_g\otimes V_h$ as the direct sum of a pair of $2$-dimensional representations induced from characters of an imaginary quadratic field. 

More precisely, recall that $g$ and $h$ are theta series of two finite order Hecke characters $\psi_g,\psi_h:G_K\longrightarrow L^\times$. As explained in the introduction, there is a decomposition \begin{equation}\label{eq: decomposition psis}
	V_g\otimes V_h\cong V_{\psi_1}\oplus V_{\psi_2}
\end{equation} where $\psi_1:=\psi_g\psi_h, \ \psi_2:=\psi_g\psi_h'$ and $V_{\psi_i}:=\Ind_{G_K}^{G_\Q}(\psi_i)$.

For $i=1,2$, $\psi_i$ is a ring class character, i.e.\ it factors through the absolute Galois group of a number field $H_i$ which is a ring class field of $K$ of some conductor $c_i$. Recall we have assumed that $p$ does not divide the conductor of $\psi_g$ and $\psi_h$, hence that the characters $\psi_1$ and $\psi_2$ are unramified at $p$, and this implies that $p\nmid c_1c_2$.  The Artin representation  $$\rho=\rho_{g}\otimes\rho_h:G_\Q\rightarrow\GL(V_g\otimes_{L_p}V_h)$$ thus factors through $\Gal(H/K)$, where $H$ is the ring class field of $K$ of conductor $c:=\mathrm{lcm}(c_1,c_2)$.
Moreover, since the prime $p$ is inert in $K$ and $p\nmid c$, the principal ideal $p\cO_K$ 
splits completely in $H$.
Fix once and for all a prime $\fp$ of the field $H$ dividing $p$ and denote $H_\fp$ the completion of $H$ at $\fp$. Then
$H_\fp=K_p$ and, in particular, the restriction of $\psi_i$ to the decomposition group $G_{K_p}=G_{H_\fp}$ is trivial for $i=1,2$.

Recall that $\Frob_p$ acts on $V_g$ with eigenvalues $\alpha_g=\lambda, \ \beta_g=-\lambda$ and on $V_h$ with eigenvalues \color{black} 
\begin{equation*}
	(\alpha_h,\beta_h)=\begin{cases}
	(-1/\lambda,1/\lambda) & \text{ if } a=+1;\\
	(1/\lambda,-1/\lambda) & \text{ if } a=-1,
	\end{cases}
\end{equation*}
so we always have $\alpha_g\alpha_h=-a$.
As in the previous sections, choose an $L$-basis of eigenvectors $v_g^\alpha, v_g^\beta$ of $V_g$ and likewise a basis $v_h^\alpha, v_h^\beta$ of $V_h$. 

Note also that $\Frob_p$ acts on $V_{\psi_1}$ and $V_{\psi_2}$ with eigenvalues $\pm1$. 

\begin{lemma}\label{lem: decomposition of Vgh and eigenvectors}
There is a $G_\Q$-equivariant isomorphism $$\Psi:V_g\otimes V_h\cong V_{\psi_1}\oplus V_{\psi_2}.$$ Moreover, $V_{\psi_i}$ admits an $L$-basis $\{ v_i^+,v_i^- \}$ of eigenvectors with relative eigenvalues $\{ +1,-1 \}$ satisfying:
	\begin{itemize}
		\item if $a=+1$, then 
		\begin{align*}
		\Psi(v_g^\alpha\otimes v_h^\alpha) = v_1^--v_2^-, \ \ \ \Psi(v_g^\beta\otimes v_h^\beta) = v_1^-+v_2^-, \ \ \ \Psi(v_g^\alpha\otimes v_h^\beta) = v_1^++v_2^+, \ \ \ \Psi(v_g^\beta\otimes v_h^\alpha) = v_1^+-v_2^+.
		\end{align*}
			\item  If $a=-1$, then 
		\begin{align*}
		\Psi(v_g^\alpha\otimes v_h^\alpha) = v_1^++v_2^+, \ \ \ \Psi(v_g^\beta\otimes v_h^\beta) = v_1^+-v_2^+, \ \ \ \Psi(v_g^\alpha\otimes v_h^\beta) =  v_1^--v_2^-, \ \ \ \Psi(v_g^\beta\otimes v_h^\alpha) = v_1^-+v_2^-.
		\end{align*}
	\end{itemize}

\end{lemma}
\begin{proof}
	Let $\psi$ be a finite order Hecke character of $K$, let $\varphi:=\theta(\psi)$ and denote $\Frob_p\in G_\Q\smallsetminus G_K$ a Frobenius element at $p$. As explained in	\cite[\S 2.2]{DLR17}, we can choose a basis $\{u_\varphi,v_\varphi\}$ of $V_\varphi$ such that, if $$\rho_\varphi: G_\Q\longrightarrow\GL(V_\varphi)$$ is the $\ell$-adic representation attached to $\varphi$, then $v_\varphi=\rho_\varphi(\Frob_p)u_\varphi$ and with respect to this basis,
	\begin{equation*}
	\rho_\varphi(\sigma)=\begin{pmatrix}
	\psi(\sigma)&0\\0&\psi'(\sigma)
	\end{pmatrix} \text{ for } \sigma\in G_K; \qquad \rho_\varphi(\tau)=\begin{pmatrix}
	0&\eta(\tau)\\ \eta'(\tau)& 0
	\end{pmatrix} \text{ for } \tau\in G_\Q\smallsetminus G_K.
	\end{equation*}
	Here ${\psi'}$ denotes the character defined by $\psi'(\sigma):=\psi(\Frob_p\sigma\Frob_p^{-1})$, and $\eta$ is a function on $G_\Q\smallsetminus G_K$. 
		Recall that the characteristic polynomial for the action of $\Frob_p$ on $V_\varphi$ is $x^2-a_p(\varphi)x+\chi_\varphi(p)=x^2-\epsilon_\psi(p),$
	where $\chi_\varphi=\chi_K\epsilon_\psi$ and $\epsilon_\psi$ is the central character of $\psi$.
	We can always choose $\{u_\varphi,  u_\varphi \}$ such that $\rho_\varphi(\Frob_p)=\begin{pmatrix}
	0&\zeta\\ \zeta&0
	\end{pmatrix},$ with $\zeta^2=\epsilon_\psi(p)$.
	Then a basis of eigenvectors for the action of $\Frob_p$ on $V_{\varphi}$ is 
	\begin{equation*}
	\{v_\varphi^{{\zeta}}=u_\varphi+v_\varphi, \qquad v_\varphi^{{-\zeta}}=u_\varphi-v_\varphi\}
	\end{equation*}
	with eigenvalue {$\zeta$, $-\zeta$} respectively.
	Using this notation we obtain the basis 
	\begin{align*}
		(u_g,v_g), \ \  (u_h,v_h), \ \  (u_1,v_1),  \ \ (u_2,v_2) 
	\end{align*}
	of $V_g$,  $V_h$, $V_{\psi_1}, V_{\psi_2}$ where we denoted $u_i:=u_{\theta(\psi_i)}$ and $v_i:=v_{\theta(\psi_i)}$. The corresponding matrix $\rho_\varphi(\Frob_p)$ is 
	\begin{align*}\label{matrix}
		\begin{pmatrix}
		0 & \lambda \\ \lambda & 0
		\end{pmatrix}, \ \  \begin{pmatrix}
		0 & 1/\lambda \\ 1/\lambda & 0
		\end{pmatrix}, \ \ \begin{pmatrix}
		0 & 1 \\ 1 & 0
		\end{pmatrix}, \ \ \begin{pmatrix}
		0 & 1 \\ 1 & 0
		\end{pmatrix}.
	\end{align*}
Consider the basis
\begin{equation*}
	(u_g\otimes u_h, u_g\otimes v_h, v_g\otimes u_h, v_g\otimes v_h)
\end{equation*}	
of 	
$V_g\otimes V_h$, and the basis
\begin{equation*}
	(u_1,v_1,u_2,v_2)
\end{equation*}  of $V_{\psi_1}\oplus V_{\psi_2}$. 
Comparing the matrices above, we conclude that the isomorphism $\Psi$ is given by the rule
	\begin{equation*}
	u_g\otimes u_h\mapsto u_1, \ \ \ u_g\otimes v_h\mapsto u_2, \ \ \ v_g\otimes u_h\mapsto v_2, \ \ \ v_g\otimes v_h\mapsto v_1.
	\end{equation*}
	Moreover, since the eigenvalues of $\Frob_p$ acting on $V_g$ and $V_h$ are $\lambda, -\lambda, 1/\lambda, 1/\lambda,$ a basis of eigenvectors for $V_g\otimes V_h$ is 
	\begin{align*}
	v_g^{{\lambda}}\otimes v^{{1/\lambda}}_h=(u_g+v_g)\otimes(u_h+ v_h), \qquad  v_g^{{\lambda}}\otimes v^{{-1/\lambda}}_h=(u_g+v_g)\otimes(u_h-v_h), \\ v_g^{{-\lambda}}\otimes v^{{1/\lambda}}_h=(u_g-v_g)\otimes(u_h+v_h), \qquad v_g^{{-\lambda}}\otimes v^{{-1/\lambda}}_h=(u_g-v_g)\otimes(u_h-v_h).
	\end{align*}
	with eigenvalues \color{black}$+1,-1,-1,+1$ \color{black}  respectively.
	On the other hand, a basis of eigenvectors for $V_{\psi_1}\oplus V_{\psi_2}$ is 
	\begin{equation*}
	(v_1^{{+}}:=u_1+v_1, \ \ \ v_1^{{-}}:=u_1-v_1, \ \ \ v_2^{{+}}:=u_2+v_2, \ \ \  v_2^{{-}}:=u_2-v_2)
 	\end{equation*}
	In conclusion,
	\begin{equation*}
	v_g^{{\lambda}}\otimes v^{{1/\lambda}}_h\mapsto v_1^{{+}}+v_2^{{+}}, \ \ v_g^{{\lambda}}\otimes v^{{-1/\lambda}}_h\mapsto v_1^{{-}}-v_2^{{-}}, \ \ v_g^{{-\lambda}}\otimes v^{{1/\lambda}}_h\mapsto v_1^{{-}}+v_2^{{-}}, \ \ v_g^{{-\lambda}}\otimes v^{{-1/\lambda}}_h\mapsto v_1^{{+}}-v_2^{{+}}.
	\end{equation*}
%

\end{proof}

Recall that 
\begin{equation*}
	V_{\psi_i}=\Ind_{G_K}^{G_\Q}(\psi_i)=\{  v:G_\Q\longrightarrow L(\psi)\mid v(\sigma\tau)=\psi_i(\sigma)v(\tau) \ \forall\sigma\in G_K, \tau\in G_\Q  \} = L \oplus L\cdot \Frob_p.
\end{equation*} The proof of the previous lemma shows that we may choose the vectors 
 $ v_i^\pm$ such that 
\begin{equation}\label{vi(1)=1}
	v_i^+(1)=v_i^-(1)=1.
\end{equation} 
Any other choice of basis would be of the form $(a_i v_i^+,b_i v_i^-)$ for some scalars $a_i,b_i\in L^\times$, and would yield the same formula for $I_p(f,g_\alpha,h_\alpha)$ up to a scalar in $L^\times$. This is fine because that is precisely the ambiguity we are working with in our framework, as explained in the introduction.

As in  previous sections, denote $$V:=V_f\otimes V_g\otimes V_h, \quad V_1:=V_f\otimes V_{\psi_1}, \quad V_2:=V_f\otimes V_{\psi_2}$$
The map $\Psi$ of Lemma \ref{lem: decomposition of Vgh and eigenvectors} induces the isomorphism 
\begin{equation}\label{eq: decomposition of V=V1+V2}
\Psi: V\overset{\cong}{\longrightarrow} V_1\oplus V_2.
\end{equation}

There is then a decomposition of the Selmer group
\begin{equation}\label{eq: decomposition of Selmer}
\Psi_*:\Sel_p(V)\overset{\cong}{\longrightarrow}\Sel_p(V_1)\oplus\Sel_p(V_2),
\end{equation}  and the analogous decomposition holds for the relaxed and the strict Selmer groups associated to $V$. By Artin's formalism, decomposition  \eqref{eq: decomposition of V=V1+V2} yields the factorization 
\begin{equation*}\label{eq: factorisation of complex Lfns}
L(E,\rho,s)=L(E,\psi_1,s)L(E,\psi_2,s)
\end{equation*}
 of classical $L$-series. As explained in \S\ref{sec: selmer groups},  Assumption \ref{ass: rank 0} on the analytic rank of $E\otimes\rho$ implies that
\begin{equation*}
\Sel_p(V_1)=\Sel_p(V_2)=0.
\end{equation*}
Denote $V_{\psi_i}^\pm$ the eigenspace of $V_{\psi_i}$ on which $\Frob_p$ acts as $\pm1$, and let 
\begin{equation*}
V_i^\pm:=V_f\otimes V_{\psi_i}^\pm.
\end{equation*}


\begin{lemma}\label{lem: H1f and H1s of Vi } 
	
	There are isomorphisms
	\begin{equation}\label{eq: H1fpm and H1spm}
		\coh_f^1(\Q_p,V_i^{\pm}) \ \cong \ \coh^1_f(K_p,V_f)^\pm\otimes V_{\psi_i}^\pm \ \text{ and } \ \coh_s^1(\Q_p,V_i^{\pm}) \ \cong \ \coh^1_s(K_p,V_f)^\pm\otimes V_{\psi_i}^\pm.
	\end{equation}
	The Bloch--Kato logarithm and dual exponential yield isomorphisms
	\begin{equation}\label{eq: bk log e exp 5.2}
			\coh_f^1(\Q_p,V_i^{\pm}) \overset{\log_{\pm}}{\longrightarrow} L_p, \qquad \coh_s^1(\Q_p,V_i^{\pm})  \overset{\exp^*_{\pm}}{\longrightarrow} L_p.
	\end{equation}
	Moreover, for $i\in\{1,2\}$, there are isomorphisms  \begin{equation}\label{eq: Sel(p)(Vi) iso con H1s}
	 \Sel_{(p)}(V_i)\overset{\partial_p}{\longrightarrow} \coh_s^1(\Q_p,V_i)\overset{\exp^*}{\longrightarrow}L_p^2.
	\end{equation}


\end{lemma}	
\begin{proof}
 The isomorphisms are obtained as in the proofs of  Lemma \ref{lem: H1(Q,V) and H(K,Vf)},  Lemma \ref{lem: tipo lem 4.1 of DR3} and Corollary \ref{cor: Sel(p) has dimension 4}, using \eqref{ah-bh}.
\end{proof}

Similarly as in Remark \ref{rem: logaritmos}, the Bloch--Kato maps \eqref{eq: bk log e exp 5.2} are related to the logarithm and the dual exponential of the representation $V_f$ as follows. Since $V_{\psi_i}$ is the Galois representation attached to the theta series $\theta(\psi_i)$, \eqref{pairing weight one} yields a pairing 

\begin{equation}\label{un altro pairing}
	\langle \ , \ \rangle:\Dderham(V_{\psi_i}^+)\times \Dderham(V_{\psi_i}^-)\longrightarrow\Dderham(L_p(\chi_{i}))
\end{equation} 
where $\chi_{i}$ is the Nebentype character of $\theta(\psi_i)$. 
Since $V_{\psi_i}$ is unramified, we have  $$\Dderham(V_{\psi_i}^\pm)=(V_{\psi_i}^\pm\otimes K_p)^{G_{\Q_p}}=V_{\psi_i}^\pm\otimes K_p^\pm.$$ 
One can introduce differentials $\omega_i^\pm\in\Dderham(V_{\psi_i}^\pm)$ as in \eqref{differentials}, and elements $\Omega_i^\pm\in K_p^\pm$ characterised by the following relation
\begin{equation*}
	v_i^\pm\otimes\Omega_i^\pm = \omega_i^\pm.
\end{equation*}
Thus \eqref{un altro pairing} in turn induces a pairing
\begin{equation*}
	\langle \ , \ \rangle: V_{\psi_i}^+ \times \Dderham(V_{\psi_i}^-)\longrightarrow L_p
\end{equation*} 
by setting $$\langle v_i^+,\omega\rangle := \dfrac{1}{\Omega_i^+}\langle \omega_i^+ ,\omega\rangle$$ for all $\omega\in\Dderham(V_{\psi_i}^-)$. 
Hence we have
\begin{equation}\label{eq: def of omegapm}
	\exp^*_\pm(x\otimes v) = \exp_{V_f}^*(x)\cdot\langle v, \omega_i^\pm\rangle
\end{equation} 
 for any $x\in \coh^1_s(K_p,V_f)^\pm$, $v\in V_{\psi_i}^\pm$.\color{black}



Recall from Lemma \ref{lem: decomposition of Vgh and eigenvectors} the basis $\{ v_i^+,v_i^-\}$ for $V_{\psi_i}$, and from Corollary \ref{cor: basis of Sel(p)} the local classes $X_\pm\in\coh_s(K_p,V_f)_{L_p}^\pm$ such that $\exp_\pm^*(X_\pm)=1$.

\begin{corollary}\label{cor: basis of sel(p)(Vi)}
	 For $i\in\{ 1,2 \}$, $\Sel_{(p)}(V_i)$ admits a basis	$
	\{ \xi_i^+, \xi_i^- \}$ characterized as
	\begin{equation*}
	\partial_p \xi_i^{\pm }=X_{\pm}\otimes v_i^\pm.
	\end{equation*}	
under the identifications given in \eqref{eq: H1fpm and H1spm}.
\end{corollary}
\begin{proof}
	The statement follows after applying Lemma \ref{lem: H1f and H1s of Vi } and using the same argument as in the proof of Corollary \ref{cor: basis of Sel(p)}.
\end{proof}

\begin{proposition}\label{prop: basis selmer V=V1+V2}
	Let 
	$$\Psi_*:\Sel_{(p)}(V)\overset{\cong}{\longrightarrow}\Sel_{(p)}(V_1)\oplus\Sel_{(p)}(V_2)$$ 
	denote the isomorphism induced by \eqref{eq: decomposition of Selmer}.

	\begin{enumerate}[$\bullet$]
		\item If $a=+1$ then 
		\begin{align*}
		\Psi_*{\xi}^{\alpha\alpha}=\xi_1^--\xi_2^-, \qquad \Psi_*{\xi}^{\alpha\beta}=\xi_1^++\xi_2^+,\qquad
		\Psi_*{\xi}^{\beta\alpha}=\xi_1^+-\xi_2^+, \qquad \Psi_*{\xi}^{\beta\beta}=\xi_1^-+\xi_2^-.
		\end{align*}
		\item If $a=-1$ then 
			\begin{align*}
		\Psi_*{\xi}^{\alpha\alpha}=\xi_1^++\xi_2^+, \qquad \Psi_*{\xi}^{\alpha\beta}=\xi_1^--\xi_2^-,\qquad
		\Psi_*{\xi}^{\beta\alpha}=\xi_1^-+\xi_2^-, \qquad \Psi_*{\xi}^{\beta\beta}=\xi_1^+-\xi_2^+.
		\end{align*}
			\end{enumerate}
 \color{black}
\end{proposition}
\begin{proof}
	Combining \eqref{eq: Sel(p)(Vi) iso con H1s} and Lemma \ref{lem: relaxed selmer iso with H1s}, we obtain a commutative diagram
	\begin{equation*}
	\xymatrix{  \Sel_{(p)}(V)\ar[r]^{\Psi_*\qquad}\ar[d]^{\partial_p} & \Sel_{(p)}(V_1)\oplus \Sel_{(p)}(V_2)\ar[d]^{\partial_p\oplus\partial_p}\\
		\coh^1_s(\Q_p,V)\ar[r]^{\Psi_*\qquad} & \coh^1_s(\Q_p,V_1)\oplus\coh^1_s(\Q_p,V_2)
	}
	\end{equation*}
	where each arrow is an isomorphism.
	Moreover, by Lemma \ref{lem: H1(Q,V) and H(K,Vf)} and Lemma \ref{lem: H1f and H1s of Vi },  we also have the commutative diagram
	\begin{equation*}
	\xymatrix{  
		\coh^1_s(\Q_p,V)\ar[r]^{\Psi_*\qquad}\ar[d] & \coh^1_s(\Q_p,V_1)\oplus\coh^1_s(\Q_p,V_2)\ar[d]\\
		(\coh^1_s(K_p,V_f)\otimes V_{gh})^{G_{\Q_p}} \ar[r]& 	(\coh^1_s(K_p,V_f)\otimes V_1)^{G_{\Q_p}}\oplus 	(\coh^1_s(K_p,V_f)\otimes V_2)^{G_{\Q_p}}
	}
	\end{equation*}
	where each arrow is an isomorphism and the bottom horizontal arrow is ${1\otimes\Psi} $. 
	By definition,
	\begin{equation*}
	\partial_p{\xi}^{\alpha\alpha}=X_{{\alpha\alpha}}\otimes v_g^\alpha\otimes v_h^\alpha\in(\coh^1_s(K_p,V_f)\otimes V_{gh}^{\alpha\alpha})^{G_{\Q_p}}\subseteq(\coh^1_s(K_p,V_f)\otimes V_{gh})^{G_{\Q_p}}.
	\end{equation*}
By \eqref{ah-bh} we have
		\begin{equation*}
			X_{\alpha\alpha}=\begin{cases}
			X_+ & \text{ if } \alpha_g\cdot\alpha_h=+1;\\
			X_- & \text{ if } \alpha_g\cdot\alpha_h=-1,
			\end{cases}
		\end{equation*}
hence we deduce
		\begin{equation*}
				(1\otimes\Psi)(\partial_p{\xi}^{\alpha\alpha})  = \begin{cases}
				X_{{\alpha\alpha}} \otimes (v_1^++v_2^+) = \partial_p\xi_1^+ + \partial_p\xi_2^+ & \text{ if } \alpha_g\cdot\alpha_h=+1\\	
				X_{{\alpha\alpha}}\otimes (v_1^--v_2^-)= \partial_p\xi_1^- - \partial_p\xi_2^- & \text{ if } \alpha_g\cdot\alpha_h=-1.
				\end{cases} 
		\end{equation*}
		Similar computations prove the remaining cases.
\end{proof}

\begin{lemma}\label{lem: tutti uguali H1f}
Let $\delta_p:E(K_p)\otimes\Q_p\longrightarrow\coh^1_f(K_p,V_f)$ denote Kummer's map.
\begin{enumerate}[$\bullet$]
	\item 
	If $a=+1$,
	there is a local point $P^+\in E(K_p)^+\otimes\Q_p$ such that 
	\begin{equation}\label{eq: resp tutti uguali -1}
	\qquad \pi_1^+\res_p(\xi_1^-) = \delta_pP^+\otimes v_1^+, \qquad  \pi_2^+\res_p(\xi_2^- ), = \delta_pP^+\otimes v_2^+ \qquad \pi_{\alpha\beta}\res_p(\xi^{\beta\beta})=\delta_pP^+\otimes v_g^\alpha\otimes v_h^\beta,
	\end{equation}
	\item If $a=-1$, 
	 there is a local point $P^-\in E(K_p)^-\otimes\Q_p$ such that
	\begin{equation}\label{eq: resp tutti uguali +1}
	\qquad\pi_1^-\res_p(\xi_1^+) = \delta_pP^-\otimes v_1^-, \qquad  \pi_2^-\res_p(\xi_2^+)= \delta_pP^-\otimes v_2^-, \qquad \pi_{\alpha\beta}\res_p(\xi^{\beta\beta})=\delta_pP^-\otimes v_g^\alpha\otimes v_h^\beta; 
	\end{equation}
\end{enumerate}

\end{lemma}

\begin{proof}
	We prove the claim when $a=-1$, 
	as the other case works similarly. By Proposition \ref{prop: basis selmer V=V1+V2},  
	\begin{equation*}
	\Psi_*\res_p(\xi^{\beta\beta})=\res_p(\xi_1^+)-\res_p(\xi_2^+)\in\coh^1(\Q_p,V_1)\oplus\coh^1(\Q_p,V_2)
	\end{equation*}
	Write
	\begin{equation*}
	\pi_{\alpha\beta}\res_p(\xi^{\beta\beta})= Q_-\otimes v_g^\alpha\otimes v_h^\beta\in\coh^1_f(K_p,V_f)^-\otimes V_{gh}^{\alpha\beta}.
	\end{equation*}		Then, by Proposition  \ref{prop: basis selmer V=V1+V2} and \eqref{eq: H1fpm and H1spm}
		\begin{equation}\label{eq: prob2}
			\Psi_*\pi_{\alpha\beta}\res_p(\xi^{\beta\beta})= Q_-\otimes(v_1^--v_2^-)=Q_-\otimes v_1^--Q_-\otimes v_2^-
	\end{equation}
	in
	\begin{equation*}
	(\coh^1_f(K_p,V_f)\otimes( V_{\psi_1}^-\oplus V_{\psi_2}^-))^{G_{\Q_p}}\cong  \coh^1_f(K_p,V_1^-)\oplus \coh^1_f(K_p,V_2^-) .
	\end{equation*}
	On the other hand, by definition of the basis $\xi_i^\pm$ and using \eqref{eq: H1fpm and H1spm},
	\begin{equation*}
	\pi_1^-\res_p(\xi_1^+)-\pi_2^-\res_p(\xi_2^+)\in \coh_f^1(\Q_p,V_1^-)\oplus \coh_f^1(\Q_p,V_2^-)\cong(\coh^1_f(K_p,V_f)\otimes( V_{\psi_1}^-\oplus V_{\psi_2}^-))^{G_{\Q_p}}.
	\end{equation*}
	More precisely, we can write
	\begin{equation}\label{eq: prob1}
	\pi_1^-\res_p (\xi_1^+)-\pi_2^-\res_p(\xi_2^+)=P_1^-\otimes v_1^- - P_2^-\otimes v_2^-
	\end{equation}
	for points $P_{i}^-\in E(K_p)^-$.
	Hence, comparing \eqref{eq: prob1} and  \eqref{eq: prob2} we deduce that $Q_-= P_{1}^-=P_2^-$.
\end{proof}

\begin{theorem}\label{prop q no esta bien}
Let $P^\pm\in E(K_p)^\pm$ be the local points of Lemma \ref{lem: tutti uguali H1f}. Then
	\begin{equation*} 
	I_p(f,g_\alpha,h_\alpha) = \dfrac{\sqrt{c}\cdot 2(1-1/p)^2\sqrt{L(E\otimes\rho,1)}}{\pi^2\langle f, f\rangle} \times\dfrac{1}{\L_{g_\alpha}}\times \log_p(P^a) \ \mod L^\times.
	\end{equation*} 
\end{theorem}
\begin{proof} Combine Theorem \ref{thm: main1} with \eqref{defPab} and Lemma \ref{lem: tutti uguali H1f}  and  use the relation between the Bloch--Kato logarithms given in \eqref{logab}. 
\end{proof}

\color{black}

\subsection{Kolyvagin classes}\label{sec: kolyvagin classes}
In this section we recall briefly the construction and main properties of the Kolyvagin classes defined in  \cite[\S6]{BD97}. 

 Assume for simplicity that  $\cO_K^\times=\{\pm1 \}$. Let $H$ be a ring class field of $K$ of conductor prime to $p$, and recall the fields $H\subseteq F_{m-1}\subseteq H(p^m)$ defined in \S\ref{sec: tate module of E and group of connected component}. For each $m$,  fix a generator $\sigma_m$ of $G_m:=\Gal(F_m/H)$ such that the image of $\sigma_m$ via the projection  
\begin{equation}\label{eq: projection Gn Gn-1} 
G_m\longrightarrow G_m/\Gal(F_m/F_{m-1})\cong G_{m-1}
\end{equation}
equals $\sigma_{m-1}$. Moreover, the elements $\{\sigma_{m}\}_m$ can be chosen so that $\sigma_{\anti}=(\sigma_{m|G_{F_{m,\fp}}})_m$, where $\sigma_{\anti}$ is the generator of $\Gamma_{\anti}$  we fixed in \S\ref{sec: the GKp cohomology of E}. 
In \cite{BD96}, Bertolini and Darmon constructed a collection of points $\{\tilde\alpha_m\in E(H(p^m))\}_m$ (in the notation of \cite{BD96}, $\tilde\alpha_m=\alpha_m(1)$) such that 
\begin{equation}\label{eq: norm alphas}
	\norm_{H(p^{m})/H}\tilde{\alpha}_m=0.
\end{equation}
For each $m\geq1$, let $$\alpha_{m}:=\norm_{H(p^{m+1})/F_{n}}\tilde\alpha_{m+1}\in E(F_{m}).$$
To each point $\alpha_{m}$, in \cite{BD97} the authors attached a Kolyvagin class $\mathbf{K}_m\in\coh^1(H,E[p^{m}])$ in the following way.
The \textit{Kolyvagin derivative operator} is the element 
$D_{m}:=\sum_{i=1}^{p^{m}-1}i\sigma_{m}^i\in\Z[G_m].$ 
\begin{lemma}
	The Kolyvagin derivative satisfies the following equality in $\Z[G_m]$ 
	\begin{equation}\label{eq: equation for kol derivative}
	(\sigma_{m}-1) D_{m}=p^{m}-\norm_{F_{m}/H}.
	\end{equation}
\end{lemma}

\begin{proof}
	Note that  $ D_{n}=-\sum_{i=1}^{p^{n}}\dfrac{{\sigma}_{n}^i-1}{{\sigma}_{n}-1}.$
	Then $({\sigma}_{n}-1 )D_{n}=-\sum_{i=1}^{p^{n}}({\sigma}_{n}-1)=-\norm_{F_{n}/H}+p^{n}.$
\end{proof}
\color{black}
Define $P_m:=D_m\alpha_m\in E(F_m)$.

\begin{lemma}\label{lem: the class Dnalphan is fixed by Gn}
	\begin{enumerate}
		\item The point $P_m\in E(F_m)$ only depends on the choice of the generator $\sigma_m$ of $G_m$ up to the multiplication by an element $a_m\in \{1,\dots,p^m-1\}$. Moreover, if we choose the elements $\{\sigma_m\}_m$ to be compatible in the sense of \eqref{eq: projection Gn Gn-1}, then $a_{m-1}\equiv a_m \mod p^{m-1}$, i.e.\ $(a_m)_m\in\varprojlim_m(\Z/p^m\Z)^\times=\Z_p^\times$.
		\item The class $[P_m]$ of $P_m$ in $E(F_m)/p^mE(F_m)$ is fixed by $G_m=\Gal(F_m/H)$.
	\end{enumerate}
\end{lemma}
\begin{proof}
	\begin{enumerate}
		\item Let $\sigma'_m$ any generator of $G_m\cong\Z/p^m\Z$ and define $D'_m:=\sum i(\sigma'_m)^i$. Then there exists an element $\overline{a_m}\in(\Z/p^m\Z)^\times$ such that $\sigma'_m=\sigma_m^{\overline{a_m}}$, where $a_m\in \{1,\dots,p^m-1\}$. Then $D'_m=\sum i\sigma_n^{\overline{a_m}i}=a_m^{-1}D_m.$
		\item It suffices to prove that the class of $D_m\alpha_m$ is fixed by the generator ${\sigma}_{m}$ of $G_m$, i.e.\ that 
	${\sigma}_{n} D_{n}{\alpha}_{n}- D_{n}{\alpha}_{n}\in p^{n}E(F_{n}).$
	Combining \eqref{eq: equation for kol derivative} and \eqref{eq: norm alphas} we obtain
		 $({\sigma}_{m}-1) D_{m}{\alpha}_{m}=p^{m}{\alpha}_{m}-\norm_{F_{m}/H}{\alpha}_{m}=p^m\alpha_m.$
	\end{enumerate}
	
\end{proof}

The inflation-restriction and Kummer exact sequences yield the following  commutative diagram
\begin{equation*}\label{eq: diagram ring class fields BD}
\xymatrix{
	& &\coh^1(F_{m}/H,E(F_m)[p^{m}]) \ar[d]^{\text{Inf}} &  & \\
	0 \ar[r] & \Big(E(H)/p^{m} E(H)\Big) \ar[r]^{\ \ \delta_p}\ar[d]^{\text{Res}} & \coh^1(H, E[p^{m}])  \ar[d]^{\text{Res}}\ar[r] & \coh^1(H,E)[p^{m}]\ar[r]  \ar[r]\ar[d]^{\text{Res}} & 0 \\
	0 \ar[r] & \Big(E(F_{m})/p^{m} E(F_{m})\Big)^{G_m} \ar[r]^{\ \ \delta_p} & \coh^1(F_{m},E[p^{m}])^{G_m} \ar[r]\ar[d] &\coh^1(F_{m},E)[p^{m}]^{G_m} & \\
	& & \coh^2(F_{m}/H,E(F_m)[p^{m}]). & & }
\end{equation*}
whose rows and columns are exact.  Let $C$ be a constant annihilating $E(F_{m})[p^{m}]$, which can be chosen  to be independent of $m$ by  \cite[Lemma 6.3]{BD97}.
By Lemma \ref{lem: the class Dnalphan is fixed by Gn}, we can consider the element $[P_m]\in(E(F_m)/p^mE(F_m))^{G_m}$, and we define
\begin{equation*}
\mathbf{K}_m:=C\mathbf{K}^{\circ}_m\in\coh^1(H,E[p^m]),
\end{equation*}
 where $\mathbf{K}^{\circ}_m$ is a preimage via $\Res$ of $-C\delta_p[P_m]$.
For each $m$, denote
\begin{equation*}
	\cG_m:=\Gal(F_m/K)\supseteq G_m:=\Gal(F_m/H)
\end{equation*}  
and fix a complete set $S=S_m=\{\sigma_1^{(m)},\dots,\sigma_h^{(m)}\}$ of representatives for the $G_m$-cosets in $\cG_m$. We require these sets to satisfy the following compatibility condition.
Let $\pi_m:\cG_m\longrightarrow\cG_{m-1}$ denote the projection induced by the inclusion $F_{m-1}\subseteq F_m$. Then 
\begin{equation*}\label{eq: compatibility of Sn}
\pi_m(\sigma^{(m)}_i)=\sigma_i^{(m-1)}
\end{equation*} for all $i$ and all $m$. In other words, under this compatibility condition we can consider a complete set $\tilde{S}=\{ \tilde{\sigma}_1,\dots,\tilde{\sigma}_m \}$ of representatives of $G_H$-cosets in $G_K$ such that, under the projection $\tilde{\pi}_m:G_K\longrightarrow\Gal(F_m/K)$ we have 
\begin{equation*}\label{eq: lift of S}
\tilde\pi_m(\tilde\sigma_i)=\sigma_i^{(m)}.
\end{equation*}

\begin{lemma}\label{lem: Cpsi} 
	The classes $\mathbf{K}_m$ form a compatible system
	\begin{equation*}
		\mathbf{K}:=(\mathbf{K}_m)_m\in\varprojlim_m \coh^1(H,E[p^m])=\coh^1(H,T_pE) \subset \coh^1(H,V_f).
	\end{equation*}
\end{lemma}
\begin{proof}
	The norm map $\norm_{F_{m+1}/F_m}:F_{m+1}\longrightarrow F_m$ induces a map
	$\norm_{F_{m+1}/F_m}:E(F_{m+1})\longrightarrow E(F_m).$
	By \cite[\S 2.45]{BD96}, we have
	\begin{equation}\label{eq: norm compatibility of alphan}
	\norm_{F_{m+1}/F_m}(\alpha_{m+1})=\alpha_m.
	\end{equation}
	Consider the composition
	\begin{equation*}
	f_{m+1}: E(F_{m+1})/p^{m+1}E(F_{m+1}) \overset{\norm_{F_{m+1}/F_m}}{\longrightarrow} E(F_m)/p^{m+1}E(F_m) \longrightarrow E(F_m)/p^mE(F_m)
	\end{equation*}
	where the second morphism is the projection given by the inclusion $p^{m+1}E(F_m)\subseteq p^mE(F_m)$.
	Hence \eqref{eq: norm compatibility of alphan} implies that 
	\begin{equation}\label{eq: compatibility alphan mod pn}
	f_{m+1}([\alpha_{m+1}] )= [\alpha_m].
	\end{equation}
	Consider the following commutative diagram
	\begin{equation*}
	\xymatrix{E(F_{m+1})/p^{m+1}E(F_{m+1}) \ar[r]^{\delta_p}\ar[d]^{f_{m+1}} & \coh^1(F_{m+1},E[p^{m+1}])\ar[d]^{f_{m+1,*}}\\
		E(F_{m})/p^{m}E(F_{m}) \ar[r]^{\delta_p}& \coh^1(F_{m},E[p^{m}]).}
	\end{equation*}
	Here $f_{m+1,*}$ is the map induced in cohomology by $f_{m+1}$, i.e.\ it is the composition 
	\begin{equation*}
	\coh^1(F_{m+1},E[p^{m+1}])\overset{\norm_{F_{m+1}/F_m}}{\longrightarrow}\coh^1(F_{m},E[p^{m+1}])\overset{p_*}{\longrightarrow}\coh^1(F_{m},E[p^{m}]),
	\end{equation*}
	where the last map is the composition with the multiplication by $p$ map $E[p^{m+1}]\overset{p}{\longrightarrow}E[p^m]$.
	Then \eqref{eq: compatibility alphan mod pn} implies that 
	\begin{equation*}
	f_{m+1,*}\delta_p[ D_{m+1}\alpha_{m+1}] = \delta_p[D_m\alpha_m].
	\end{equation*}
\end{proof}

Recall the groups  $\Phi_{m,\fp}$ and $\Phi_{\infty,\fp}$ defined in \S\ref{sec: tate module of E and group of connected component}.
By Lemma \ref{lemma: 2} and Proposition \ref{prop:1}, there is an injection 
\begin{equation*}
	\hom_{\cont}(\Gamma_{\anti},\Phi_{\infty,\fp}\otimes\Q_p)\subseteq\coh^1_s(K_p,V_f).
\end{equation*}
Recall the choice $\sigma_m$ of generator of $G_m$; if we still denote $\sigma_m$ its restriction to $\Gal(F_{m,\fp}/H_\fp)=\Gal(F_{m,\fp}/K_p)$, then $\Gamma_{\anti}$ is generated by $\sigma_{\anti}:=(\sigma_m)_m$.

 Recall the map $\partial_p: \coh^1(H,V_f) \ra \coh^1_s(K_p,V_f)$ onto the singular quotient.

\begin{proposition}\label{prop: C}
	The class $\mathbf{K}$ lies in $\Sel_{(p)}(H,V_f)$ and $\partial_p  \mathbf{K}\in 	\hom_{\cont}(\Gamma_{\anti},\Phi_{\infty,\fp}\otimes\Q_p)$. Moreover, if  $\bar{\alpha}_m$ denotes the image of $\alpha_m$ in $\Phi_{m,\fp}$ and  $\bar{\alpha}:=(\bar{\alpha}_m)_m\in\Phi_{\infty,\fp}$, then
	\begin{equation*}
	\partial_p \mathbf{K}(\sigma_\anti) = \bar{\alpha}.
	\end{equation*}
\end{proposition}
\begin{proof}
\cite[Proposition 6.9, 1,2]{BD97}.
\end{proof}

\subsection{$I_p(f,g_\alpha,h_\alpha)$ and Kolyvagin classes}\label{sec: Lpg and kolyvagin classes}
Recall the ring class characters $\psi_1,\psi_2$ appearing in the decomposition \eqref{eq: decomposition psis}. Since they are ring class characters unramified at $p$, they factor through the Galois group $\Gal(H/K)$, where $H$ is a ring class field of conductor prime to $p$.  

Using the Kolyvagin classes described in  \S\ref{sec: kolyvagin classes}, in this section we define elements $\mathbf{K}^{\psi_i}$ in the relaxed Selmer groups  $\Sel_{(p)}(K, V_f\otimes\psi_i)$ for $i=1,2$. The aim of this section is to compare the classes $\mathbf{K}^{\psi_1}, \mathbf{K}^{\psi_2}$ with the local points in $E(K_p)^\pm$ appearing in Theorem \ref{prop q no esta bien}, in order to obtain a formula for the value $I_p(f,g_\alpha,h_\alpha)$ in terms of these Kolyvagin classes.
Using the notation of the previous sections, consider the dual exponential 
\begin{equation*} 
\exp^*_\pm:\coh^1_s(K_p,V_f)^\pm\longrightarrow\Q_p
\end{equation*}
of \eqref{eq: exp pm onto Qp}.
As explained in \S\ref{sec: ex appendix},   Tate's uniformisation induces the isomorphism  $\varphi:\Phi_{\infty,\fp}\overset{\cong}{\longrightarrow}\Z_p$
 (in the notation of Lemma \ref{lem: Phi infty and limit of bar Phi1 + Phi infty and limit of bar Phi2 + rem: 2}, $\varphi:=\overline{\varphi}_{\Tate}$ maps $\bar{q}$ to $1$). 
Recall the point  $\bar{\alpha}\in\Phi_{\infty,\fp}$ defined in Proposition \ref{prop: C}, the sign
$$a:=a_p(E)\in \{ \pm 1\}$$ and the period
\begin{equation*}
	\Pi_p:=\varphi(\bar{\alpha})\in\Z_p.
\end{equation*}

\begin{proposition}\label{prop: exp partialp C} We have
$$\partial_p  \mathbf{K}\in\coh^1_s(K_p,V_f)^{-a}\quad \mbox{ and } \quad \exp^*_{-a}\partial_p\mathbf{K} =\dfrac{\ord_p(q_E)}{p}\Pi_p.$$
In particular $$\res_p(\mathbf{K})^a := \res_p(\mathbf{K})+a\res_p(\mathbf{K}^{\Frob_p})$$ is cristalline, i.e.\,lies in $\coh^1_f(\Q_p,V_f)^a$.
\end{proposition}
\begin{proof}
	Recall  we can regard $
	\hom_{\cont}(\Gamma_{\anti},\Phi_{\infty,\fp}\otimes\Q_p)$ as a subspace of $\coh^1_s(K_p,V_f)$. By  \eqref{eq: varphi pm}  $\Frob_p$ acts on $
	\hom_{\cont}(\Gamma_{\anti},\Phi_{\infty,\fp} \otimes\Q_p)$
	as multiplication by $-a$. Hence $ \res_p(\mathbf{K})$ belongs to $\coh^1_s(K_p,V_f)^{-a}$ by Proposition \ref{prop: C}, and the formula for $\exp^*_{-a}\partial_p\mathbf{K}$ follows from Corollary \ref{cor: main appendix}.
\end{proof}




For $i=1,2$ let $$\Tr_i = \sum_{\sigma\in\tilde S}\psi_i(\sigma)\sigma: \coh^1(H,V_f)\longrightarrow\coh^1(K,V_f\otimes\psi_i)$$ denote the trace map onto the $\psi_i$-isotypic component
and consider the isomorphism given by Shapiro's Lemma
\begin{equation}\label{eq: shapiro global}
	\Sh:\coh^1(K,V_f\otimes\psi_i)\overset{\cong}{\longrightarrow}\coh^1(\Q,V_f\otimes\Ind_{G_K}^{G_\Q}(\psi_i)) =: \coh^1(\Q,V_i).
\end{equation}
Define
\begin{equation*}
	\mathbf{K}^{\psi_i}:=\Sh(\Tr_i(\mathbf{K}))\in\coh^1(\Q,V_i).
\end{equation*}
It follows from Proposition \ref{prop: C} that $\mathbf{K}^{\psi_i}$ lies in $\Sel_{(p)}(V_i)$.


Frobenius element $\Frob_p\in G_{\Q_p}$ acts on $\coh^1(K_p,V_f)$  as an involution and we may consider the decomposition in $\pm$-eigenspaces
\begin{equation*}
\coh^1(K_p,V_f\otimes L_p)=\coh^1(K_p,V_f\otimes L_p)^+\oplus\coh^1(K_p,V_f\otimes L_p)^-.
\end{equation*}

As one readily verifies, Shapiro's isomorphism restricts to
\begin{equation}\label{lem: shapiro and eigenspaces}
\Sh_p:\coh^1(K_p,V_f\otimes L_p)^\pm\overset{\cong}{\longrightarrow}\coh^1(\Q_p,V_f\otimes V_{\psi_i}^\pm).
\end{equation}


\begin{corollary}\label{prop: fine} The Kolyvagin class
	$\mathbf{K}^{\psi_i}$ satisfies
	\begin{equation*}
	\mathbf{K}^{\psi_i} = 	\dfrac{h\ord_p(q_E)\Pi_p}{ p}\cdot\xi_i^{-a}.
	\end{equation*}
\end{corollary}

\begin{proof}
	Since $\psi_{i|G_{K_p}} = 1$, the restriction of $\Tr_i$ to $\coh^1(H_\fp,V_f\otimes L_p)=\coh^1(K_p,V_f\otimes L_p)$ is  multiplication by $h=[H:K]$. Hence $ \res_p(\mathbf{K}^{\psi_i})=h\cdot \Sh_p( \res_p(\mathbf{K}))$. 
	
	By Proposition \ref{prop: exp partialp C} we have $\partial_p  \mathbf{K}\in\coh^1_s(K_p,V_f)^{-a}$, and by  \eqref{lem: shapiro and eigenspaces}, it follows that $\partial_p  \mathbf{K}^{\psi_i}\in\coh^1_s(\Q_p,V_{\psi_i}^{-a}) $.   Recall the choice of basis $\xi_i^{\pm}$ made in Corollary \ref{cor: basis of sel(p)(Vi)}. We may thus write 
	\begin{equation*}\label{eq: partial shap}
	\mathbf{K}^{\psi_i}=\dfrac{\exp^*_{-a}(\partial_p  \mathbf{K}^{\psi_i})}{\exp^*_{-a}(\partial_p \xi_i^{-a})}\cdot\xi_i^{-a}.
	\end{equation*}
	By definition of the basis $\{ \xi_i^+, \xi_i^-   \}$ and by \eqref{eq: def of omegapm}, the denominator in the above expression is $\exp^*(X_{-a})\langle v_i^{-a},\omega_i^{-a}\rangle = \langle v_i^{-a},\omega_i^{-a}\rangle $.
Let $$R\otimes v_i^{-a}\in \coh^1_s(K_p,V_f)\otimes V_{\psi_i}^{-a}$$ denote the image of $\partial_p\mathbf{K}^{\psi_i}$ via the isomorphism	\eqref{eq: H1fpm and H1spm}. Then 
\begin{equation*}
	\mathbf{K}^{\psi_i}= \dfrac{\exp^*(R) \langle v_i^{-a},\omega_i^{-a}\rangle }{ \langle v_i^{-a},\omega_i^{-a}\rangle }\cdot \xi_{i}^{-a} = \exp^*(R)\cdot \xi_{i}^{-a} .
\end{equation*}
In order to compute the dual exponential of $R$, we need the following explicit expression for \eqref{eq: shapiro explicit} for $\Sh_p$.
Recall that, as a $L_p[G_{\Q_p}]$-module,
\begin{equation*}
V_{\psi_i}=\Ind_{G_{K_p}}^{G_{\Q_p}}(1)=\{ v:G_{\Q_p}\longrightarrow L_p \mid v(\sigma\tau)=v(\tau) \ \forall\sigma\in G_{K_p}, \tau\in G_{\Q_p}  \}.
\end{equation*}
Consider the map 
\begin{equation*}\label{ev}
ev: V_{\psi_i}\longrightarrow L_p, \qquad v\mapsto v(1).
\end{equation*}
It is an equivariant $G_{K_p}$-morphism which is compatible with the inclusion $G_{K_p}\hookrightarrow G_{\Q_p}$, so it induces a morphism
\begin{equation}\label{eq: shp-1}
\Sh_p^{-1}:\coh^1(\Q_p,V_f\otimes V_{\psi_i})\longrightarrow \coh^1(K_p,V_f\otimes L_p)=\coh^1(K_p,V_f\otimes\psi_i),
\end{equation}
which is the inverse of Shapiro's isomorphism \eqref{eq: shapiro global} restricted to $G_{K_p}$.
More explicitly, if $\xi:G_{\Q_p}\longrightarrow V_f\otimes V_{\psi_i}$ represents a class in $\coh^1(\Q_p,V_f\otimes V_{\psi_i})$, then 
\begin{equation}\label{eq: shapiro explicit}
\Sh_p^{-1}(\xi):=(\id\otimes ev)\circ\xi_{|G_{K_p}}
\end{equation}
represents its image via \eqref{eq: shp-1}.
Recall that $R\otimes v_i^{-a}$ is the image of $h\cdot \partial_p\mathbf{K}$ via the composition
\begin{equation*}
\coh^1_s(K_p,V_f\otimes L_p)^{-a}\overset{\Sh_p}{\longrightarrow}\coh^1_s(\Q_p,V_f\otimes V_{\psi_i}^{-a})\overset{\Res}{\longrightarrow}\coh^1_s(K_p,V_f\otimes V_{\psi_i}^{-a})^{G_{\Q_p}}\cong \coh^1_s(K_p,V_f)^{-a}\otimes V_{\psi_i}^{-a}.
\end{equation*}
We conclude that the element  $R\otimes v_i^{-a}$ is represented by the cocycle
\begin{align*}
h\cdot\Res(\Sh_p(\partial_p\mathbf{K})):G_{K_p}\longrightarrow V_f\otimes V_{\psi_i}^{-a}
\end{align*}
satisfying 
\begin{equation*}
	R\cdot( v_i^{-a}(1)) =h\cdot \partial_p\mathbf{K}.
\end{equation*}
In other words, by \eqref{vi(1)=1}, we have
\begin{equation*}
	\exp^*(R) =
	h\cdot\exp^*(\partial_p \mathbf{K}).
\end{equation*}
The statement follows by applying Proposition \ref{prop: exp partialp C}.
\end{proof}

\begin{corollary}\label{cor: p inert explicit version nonsplit}
	We have
	\begin{equation*} 
I_p(f,g_\alpha,h_\alpha) = \dfrac{\sqrt{c} \cdot 2p(1-1/p)^2}{\ord_p(q_E)}\cdot\dfrac{\sqrt{L(E\otimes\rho,1)}}{\pi^2\langle f, f\rangle\Pi_p}
\cdot\dfrac{1}{\L_{g_\alpha}}\times \log_p(Q_p^a) \ \mod L^\times,
\end{equation*} 
	where $Q_p^{a}\in E(K_p)^{a}$ is characterized by  $\delta_p(Q_p^{a})= \res_p(\mathbf{K})^a\in\coh^1_f(K_p,V_f)^{a}$. 
\end{corollary}
\begin{proof}
	By \eqref{lem: shapiro and eigenspaces},  if  $\pi_{a}:\coh^1(\Q_p,V_i)\longrightarrow\coh^1(\Q_p,V_i^{a})$ denotes the natural projection, then
	\begin{equation}\label{lab}
	\pi_{a}\res_p(\mathbf{K}^{\psi_i}) = h\cdot \Sh_p(\res_p(\mathbf{K})^a)\in\coh^1_f(\Q_p,V_i^{a}). 
	\end{equation}
	Arguing as in the proof of Corollary \ref{prop: fine}, write 
	\begin{equation*}
		A\otimes v_i^a\in \coh^1_f(K_p,V_f)^a\otimes V_{\psi_i}^a
	\end{equation*}
	for the image of \eqref{lab} via the isomorphism \eqref{eq: H1fpm and H1spm}. Then 
	\begin{equation*}
		A= \dfrac{h\cdot \res_p(\mathbf{K})^a}{v_i^a(1)}=h\cdot \res_p(\mathbf{K})^a.
	\end{equation*}
Combining this with Corollary \ref{prop: fine} and Lemma \ref{lem: tutti uguali H1f}, we obtain
\begin{equation*}
h\cdot\res_p(\mathbf{K})^a\otimes v_i^a = \pi_a\res_p(\mathbf{K}^{\psi_i}) = \dfrac{h\ord_p(q_E)\Pi_p}{p }\delta_p(P^{a})\otimes v_i^a.
\end{equation*}
Then
\begin{equation*}
 \delta_p(P^{a})= 	\dfrac{p }{\ord_p(q_E)\Pi_p} \cdot \res_p(\mathbf{K})^a,
\end{equation*}
and the thesis follows by applying Theorem	\ref{prop q no esta bien}.
	
\end{proof}

\color{black}

\bibliographystyle{alpha}
\bibliography{refs2}

\end{document}